\documentclass[generic,preprint]{imsart}
\usepackage{natbib}
\usepackage[T1]{fontenc}
\usepackage{a4wide, url}
\usepackage[english]{babel}
\usepackage{graphicx}
\usepackage{array}
\usepackage{multirow}
\usepackage{amsthm}
\usepackage{amsmath}
\usepackage{amsfonts}
\usepackage{amssymb}
\usepackage{pgf}
\usepackage{bm}
\usepackage{paralist}
\usepackage{pdflscape}
\usepackage{lscape,amssymb, amsmath,verbatim,graphicx}
\usepackage{epsfig,subfigure,float,subfigure}
\usepackage{graphicx}
\usepackage{epsfig}
\usepackage{amssymb}
\usepackage{amsmath}
\usepackage{amsfonts}
\usepackage{geometry}
\usepackage{scalefnt}
\usepackage{graphicx}
\usepackage{array}
\usepackage{multirow}
\usepackage{amsthm}
\usepackage{amsmath}
\usepackage{amsfonts}
\usepackage{amssymb}
\usepackage{pgf}
\usepackage{bm}
\usepackage{paralist}
\usepackage{pdflscape}
\usepackage{rotating}
\usepackage{scalefnt}

\DeclareMathAlphabet\mathbfcal{OMS}{cmsy}{b}{n}
\RequirePackage[OT1]{fontenc}
\RequirePackage{amsthm,amsmath}
\RequirePackage[colorlinks,citecolor=blue,urlcolor=blue]{hyperref}

\newcommand{\beq}{\begin{equation}}
\newcommand{\eeq}{\end{equation}}
\newcommand{\bea}{\begin{eqnarray}}
\newcommand{\eea}{\end{eqnarray}}
\newcommand{\bit}{\begin{itemize}}
\newcommand{\eit}{\end{itemize}}
\newcommand{\ben}{\begin{enumerate}}
\newcommand{\een}{\end{enumerate}}
\newcommand{\bpm}{\begin{pmatrix}}
\newcommand{\epm}{\end{pmatrix}}
\newcommand{\bbm}{\begin{bmatrix}}
\newcommand{\ebm}{\end{bmatrix}}

\newtheorem{theorem}{Theorem}
\theoremstyle{plain}

\newtheorem{corollary}{Corollary}

\newtheorem{lemma}{Lemma}

\newtheorem{assumption}{Assumption}
\numberwithin{equation}{section}
\graphicspath{{Figures/}}
\begin{document}

\begin{frontmatter}
\title{Testing for strict stationarity in a random
coefficient autoregressive model}
\runtitle{Testing for stationarity}

\begin{aug}
\author{\fnms{Lorenzo} \snm{Trapani}$^{*}$\ead[label=e2]{lorenzo.trapani@nottingham.ac.uk}}
\runauthor{L. Trapani}

\affiliation{University of Nottingham\thanksmark{m2}}

\address{$^{*}$University of Nottingham\\
\printead{e2}\\
}
\end{aug}

\begin{abstract}
We propose a procedure to decide between the null hypothesis of (strict) stationarity and the alternative of non-stationarity, in the
context of a Random Coefficient AutoRegression (RCAR). The procedure is based on randomising a diagnostic which
diverges to positive infinity under the null, and drifts to zero under the
alternative. Thence, we propose a randomised test which can be used directly and - building on it - a decision rule to discern between the null and the alternative. \
The procedure can be applied under very general circumstances: albeit
developed for an RCAR model, it can be used in the case of a standard AR(1)
model, without requiring any modifications or prior knowledge. Also, the
test works (again with no modification or prior knowledge being required) in
the presence of infinite variance, and in general requires minimal
assumptions on the existence of moments.

\textbf{Keywords}: Random Coefficient AutoRegression, Stationarity, Unit
Root, Heavy Tails, Randomised Tests.

\textbf{AMS 2000 subject classification}: Primary 62F05; secondary 62M10.
\end{abstract}

\renewcommand{\thefootnote}{$\ast$} \thispagestyle{empty}


\renewcommand{\thefootnote}{\arabic{footnote}}
\end{frontmatter}

\newpage

\newpage

\section{Introduction\label{introduction}}

In this paper, we propose a procedure to decide in favour of, or against,
the strict stationarity of a series generated by a Random Coefficient
AutoRegressive (RCAR) model:%
\begin{equation}
X_{t}=\left( \varphi +b_{t}\right) X_{t-1}+e_{t}\text{, \ \ \ }1\leq
t<\infty ,\quad \mbox{where}\;\;X_{0}\;\;\mbox{is an initial value}.
\label{rca}
\end{equation}

Model (\ref{rca}) has been paid considerable attention by the literature,
mainly due to its flexibility and analytical tractability - see %
\citet{nichollsquinn} and the references in \citet{aue2011}, and also the
article by \citet{diaconis} where several examples are discussed. Equation (%
\ref{rca}) has also become increasingly popular in econometrics. Indeed, it
is immediate to see that (\ref{rca}) nests the AR(1) model as a special
case, with the advantage that it can be viewed as a competitor for a model
with an abrupt break in the autoregressive root (see, especially, a related
paper by \citealp{gky}). Further, a closely related specification is the
so-called Double AutoRegressive (DAR) model $X_{t}=\varphi X_{t-1}+v_{t}$
with $v_{t}=\sqrt{a+bX_{t-1}^{2}}\epsilon _{t}$ and $\epsilon _{t}$ is an 
\textit{i.i.d.} process; this model, in turn, nests the popular ARCH
specification (see Theorem 1 in \citealp{tsay1987}). Finally, (\ref{rca})
has been employed as a more general alternative to deterministic unit root
processes, with $\varphi =1$ and $b_{t}$ not identically zero - this is
known in the literature as the Stochastic Unit Root (STUR) process, and we
refer to the contributions by \citet{granger1997}, \citet{mccabe1995} and %
\citet{leybourne1996}, among others, for an overview.

Various aspects of the inference on (\ref{rca}) are well developed. The
estimation of $\varphi $, in particular, has been studied in numerous
contributions, for both the stationary case (\citealp{aue06}) and the
nonstationary case (\citealp{berkes2009}); \citet{aue2011} suggest using the
Quasi-Maximum Likelihood (QML) estimator, showing that the estimator of $%
\varphi $ is always consistent and asymptotically normal, irrespective of
the stationarity (or lack thereof) of $X_{t}$, as long as $b_{t}$ is not
equal zero almost surely. The same result has also been shown for several
other estimators, like the Weighted Least Squares (WLS) estimator (see %
\citealp{HT2016}) and the Empirical Likelihood (EL) estimator, with %
\citet{hillpeng2014} and \citet{hillpeng2016} showing that standard normal
inference holds for all possible cases.

Conversely, other parts of the inference on (\ref{rca}) are not fully
established. In particular, only few results are available as far as testing
for the stationarity/ergodicity of $X_{t}$ is concerned. Most contributions
focus on restricted versions of (\ref{rca}), e.g. testing whether $X_{t}$ is
a genuine unit root process versus the alternative of a STUR process - see %
\citet{mccabe1995}, \citet{leybourne1996}, \citet{distaso2008} and %
\citet{nagakura}. Indeed, such approaches are not fully clear about $X_{t}$
being stationary or not, since a STUR\ process may well be strictly
stationary (we refer to \citealp{yoon} for an insightful discussion). The
case of a stochastic unit root is of great interest, since it is a marked
departure from the usual deterministic unit root case. Indeed, whilst the
unit root hypothesis may often hold for several series, nonstationarity
itself may not; in this respect, the literature has explored, especially in
the context of financial data, the notion of \textquotedblleft volatility
induced stationarity\textquotedblright , where a series may have a unit
root, but its variance grows to infinity so as to compensate the deviations
from the mean (see \citealp{rahbek} and \citealp{kanaya}). There are several
reasons why it would be useful to find out whether $X_{t}$ is stationary and
ergodic or not: estimating the variance of the error term $e_{t}$ is
possible only under stationarity; further, in order to recover standard
normal inference when $X_{t}$ is nonstationary through e.g. the QML\
estimator discussed in \citet{aue2011}, it is required that $b_{t}$ is not
equal to zero almost surely, otherwise the standard OLS estimator has a
faster rate of convergence (\citealp{wang2015}). Although there is a
plethora of tests for a unit root in the AR(1) case, to the best of our
knowledge there are no contributions which have a sufficient amount of
generality to be applicable to the context of AR(1) and RCAR(1) models. In
addition to this, even in the AR(1) case, testing procedures usually require
the existence of at least the first two moments; tests for a unit root in
presence of infinite variance have been studied (\citealp{phillips1990}),
but their implementation usually requires using a different limiting
distribution (and, consequently, different critical values) - note however
that, in a related contribution, \citet{cavaliere2016} provide a solution by
using the bootstrap, but their results do not cover the RCAR\ case and it is
not immediately clear how to generalise the bootstrap to the RCAR\ model
under general conditions (see \citealp{fink2013}).

Such a lack of procedures to check for the (strict) stationarity of an RCAR
model is undesirable, in the light of the empirical potential of this family
of models. Indeed, the literature has developed a fair amount of tests for
stationarity (either as the null, or, perhaps more frequently, the
alternative hypothesis) in various contexts. There are, as is well-known,
many contributions for the unit root problem in the case of a linear model
(such as the AR(1), which as noted above is nested in the RCAR set-up), but
there are also quite a few contributions for more general frameworks. For
example, \citet{kapetanios2007}, \citet{limaneri} and \citet{harveybusetti}
propose tests for the null of stationarity by using quantile-based
techniques, thus not needing to specify a model for the data. In addition,
there are also several tests that deal with specific nonlinear models, such
as TAR (\citealp{tsay1997}, \citealp{caner}), STAR (\citealp{kapetanios2003}%
) and ESTAR (\citealp{kilic}). In particular, building on the notion of top
Lyapunov exponent, \citet{guo} propose a test for the null of strict
stationarity in the context of a DAR model; see also \citet{linton}, and %
\citet{ling2004} and \citet{francq2012}. However, all the contributions
cited above have, in common, the assumption that some moments of the
distribution of the data (typically, the variance) need to exist.

\bigskip

This paper extends the literature in at least three directions. Firstly, as
mentioned above, contributions on testing for the stationarity of the RCAR\
model are rare; furthermore, the existing exceptions (e.g. \citealp{zhao2012}%
) typically require moment restrictions such as having finite variance.
Furthermore, our test, which is constructed for the null of (strict)
stationarity in the RCAR framework, can also be applied to the case of an
AR(1) model with no modifications and no need to test for genuine randomness
in (\ref{rca}). Indeed, even though in the main part of the paper we
consider the construction of a test for stationarity without having any
deterministics in (\ref{rca}), as is typical in the case of infinite
variance, we also consider extensions to include these: our test can be
applied without any pre-treatment of the data in the presence of bounded
deterministics (including constant and piecewise constant functions) and,
subject to detrending, to the case of trends. We also propose, as a
by-product, a test for the null of non-stationarity with the same level of
generality as discussed above. Finally, our test can be applied in the
presence of infinite variance and even infinite mean (indeed, all that is
required is the existence of some moments, i.e. $E\left\vert
e_{0}\right\vert ^{\epsilon }<\infty $ and $E\left\vert b_{0}\right\vert
^{\epsilon }<\infty $, for arbitrarily small $\epsilon $); again, the test
does not require any prior knowledge as to the existence of moments, and it
can be applied directly with no modifications (even in this case,
irrespective of having an RCAR or an AR(1) model). Having a test which can
be readily applied in the presence of heavy tails with no need to estimate
nuisance parameters is arguably useful in empirical applications: data with
infinite variance (or even infinite mean) could occur in applications to
finance and also to other disciplines - we refer to the textbook by %
\citet{embrechts} for details and discussed examples. An important feature
of several inferential procedures is that they rely on the estimation of the
so-called \textquotedblleft tail index\textquotedblright , which
characterizes the largest existing moment of a random variable: this
parameter is however notoriously difficult to estimate (see %
\citealp{embrechts}). Note also that, since we allow for $E\left\vert
X_{t}\right\vert ^{2}=\infty $, our focus is on strict, as opposed to weak,
stationarity.

From a technical point of view, we construct a scale-invariant statistic
which diverges to positive infinity under the null of stationarity, and
drifts to zero at a polynomial rate under the alternative hypothesis, in all
possible circumstances. In the context of a non-linear model like the RCAR
(also with the possibility of having infinite variance or even infinite
mean) it is not easy to construct a statistic which converges to a unique
and \textquotedblleft easy\textquotedblright\ limiting distribution: hence,
we do not derive any distributional limit for our test statistic, only rates
of convergence/divergence. Given that our proposed statistic does not have a
usable limiting distribution, we propose to randomise it, in a similar
spirit to \citet{corradi2006} and \citet{bandi2014}. The test can then be
employed as is; however, given that it is a randomised test whose outcome
depends on the auxiliary randomisation, we complement our methodology by
also proposing a strong decision rule which is independent of the
randomisation, thereby giving the same outcome to different users.

\bigskip

The paper is organised as follows. The main assumptions, the test statistic,
and the asymptotics, are all in Section \ref{test}. In Section \ref%
{discussion} we report extensions to: including deterministics, testing for
the null of non-stationarity, and developing a family of related statistics.
Monte Carlo evidence is reported in Section \ref{montecarlo}, where we also
carry out an empirical application to illustrate the use of our procedure
and in particular of the proposed strong decision rule. Section \ref%
{conclusions} concludes. Technical results and all proofs are in Appendix.

NOTATION We use \textquotedblleft\ $\rightarrow $ \textquotedblright\ to
denote the ordinary limit; \textquotedblleft a.s.\textquotedblright\ stands
for almost sure (convergence); $C_{0}$, $C_{1}$,... denote positive and
finite constants that do not depend on the sample size (unless otherwise
stated), and whose value may change from line to line; $I_{A}\left( x\right) 
$ is the indicator function of a set $A$; strictly positive, arbitrarily
small constants are denoted as $\epsilon $ - again, the value of $\epsilon $
may change from line to line. Finally, since all results in the paper hold
almost surely, orders of magnitude for an a.s. convergent sequence (say $%
s_{T}$) are denoted as $O\left( T^{\varsigma }\right) $ and $o\left(
T^{\varsigma }\right) $ when, for some $\epsilon >0$ and $\tilde{T}<\infty $%
, $P\left[ \left\vert T^{-\varsigma }s_{T}\right\vert <\epsilon \text{ for
all }T\geq \tilde{T}\right] =1$ and $T^{-\varsigma }s_{T}\rightarrow 0$
a.s., respectively.

\section{Testing for strict stationarity\label{test}}

This section contains all the relevant theory. In Section \ref%
{classification} we spell out the necessary and sufficient conditions
required for strict stationarity, and discuss under which circumstances the
second moment of $X_{t}$ is finite. Assumptions are in Section \ref%
{assumption}, and the test statistic is reported in Section \ref{testing}.

\subsection{Classification and hypothesis testing framework\label%
{classification}}

Recall (\ref{rca})%
\begin{equation*}
X_{t}=\left( \varphi +b_{t}\right) X_{t-1}+e_{t}\text{.}
\end{equation*}%
It is well known that, under minimal assumptions such as the existence of
logarithmic moments for $e_{0}$ and $\varphi +b_{0}$, three separate regimes
can hold for the solutions of (\ref{rca}), depending on the value taken by $%
E\ln \left\vert \varphi +b_{0}\right\vert $:

\noindent \textit{(i)} If $-\infty \leq E\ln \left\vert \varphi
+b_{0}\right\vert <0$, then $X_{t}$ converges exponentially fast (for all
initial values $X_{0}$) to%
\begin{equation}
\bar{X}_{t}=\sum_{s=-\infty }^{t}e_{s}\prod_{z=s+1}^{t}\left( \varphi
+b_{z}\right) .  \label{stationary}
\end{equation}%
Note that, when $Eb_{0}^{2}>0$, $E\ln \left\vert \varphi +b_{0}\right\vert $
can be negative even when $\varphi =1$: thus, the STUR\ process can converge
to a strictly stationary solution, although in such a case $\bar{X}_{t}$ has
an infinite second moment (see \citealp{hwangbasawa}). More generally, the
variance of $\bar{X}_{t}$ needs not be finite under strict stationarity; a
necessary and sufficient condition for this is $Eb_{0}^{2}+\varphi ^{2}<1$ (%
\citealp{quinn1982}).

\noindent \textit{(ii)} If $E\ln \left\vert \varphi +b_{0}\right\vert >0$,
then $X_{t}$\ is nonstationary and it exhibits an explosive behaviour. This
case has also been studied in depth by the literature: Berkes \textit{et al.}%
\ (2009) show that $\left\vert X_{t}\right\vert \rightarrow \infty $
exponentially fast.\newline
\noindent \textit{(iii) } In the boundary case $E\ln \left\vert \varphi
+b_{0}\right\vert =0$, $X_{t}$\ is nonstationary (see also the comments
after Assumption \ref{as-2}). Even in this case $\left\vert X_{t}\right\vert 
$ diverges, but at a slower rate than exponential. This case has been paid
comparatively less attention in the literature.\newline
Clearly, this classification also holds for the basic AR(1) model, i.e.\ for
the case $b_{0}=0$ a.s.

\bigskip

On the grounds of the classification above, we propose a procedure to decide
between%
\begin{equation}
\left\{ 
\begin{tabular}{ll}
$H_{0}:$ & $X_{t}\text{ is strictly stationary}$ \\ 
$H_{A}:$ & $X_{t}\text{ is nonstationary}$%
\end{tabular}%
\right.  \label{fmwk}
\end{equation}

\subsection{Assumptions\label{assumption}}

We now introduce and discuss the main assumptions. The first assumption must
be satisfied by $X_{t}$ irrespective of the regime it belongs to, and it can
be compared to the assumptions in \citet{aue06}.

\begin{assumption}
\label{as-1} It holds that: (i) $\{b_{t},-\infty <t<\infty \}$ and $\left\{
e_{t},-\infty <t<\infty \right\} $ are independent sequences; (ii) $%
\{b_{t},-\infty <t<\infty \}$ are independent and identically distributed
random variables; (iii) $\{e_{t},-\infty <t<\infty \}$ are independent and
identically distributed random variables; (iv) ~$b_{0}$ and $e_{0}\ $are
symmetric random variables; (v) ~$E\left\vert b_{0}\right\vert ^{\nu
}<\infty $ and $E\left\vert e_{0}\right\vert ^{\nu }<\infty $ for some $\nu
>0$; (vi) $X_{0}$ is independent of $\left\{ e_{t},b_{t},t\geq 1\right\} $
with $E\left\vert X_{0}\right\vert ^{\nu }<\infty $.
\end{assumption}

Assumption \ref{as-1} contains minimal requirements as far as the existence
of moments is concerned, and, in this respect, it is very general. Note that
by part \textit{(iv)} of the assumption, we require that, when the mean of $%
e_{t}$ and $b_{t}$\ exists, this is zero. As far as the \textit{i.i.d.}
requirement for $e_{t}$ and $b_{t}$ is concerned, it is typical in this
literature (see \citealp{aue06}), and it is imposed only in order for the
main arguments in the proofs not to be overshadowed by technical details.
Indeed, relaxing the assumption of independence is possible: the conditions
for stationarity mentioned above hold as long as $\left\{
e_{t},b_{t}\right\} $ is strictly stationary and ergodic (see Theorem 4.1 in %
\citealp{moulines}) as well as having logarithmic moments. Also, the
technical arguments used in the paper (essentially, the ergodic theorem, the
SLLN, and the almost sure Invariance Principle) can all be extended to the
case of weakly dependent data. A major advantage of our approach, in this
case, is that our test statistic is based only on rates and therefore, even
in the presence of dependence, our test statistic would not require any
modifications such as, for example, the estimation of long run variance
matrices.\newline

Stationary units must also satisfy the following assumption.

\begin{assumption}
\label{as-2} If $E\ln \left\vert \varphi +b_{0}\right\vert <0$, it holds
that (i) $P\left( |\bar{X}_{0}|=0\right) <1$; (ii) $P\left( \left\vert
e_{0}\right\vert =0\right) <1$.
\end{assumption}

Assumption \ref{as-2} is also quite standard in the literature. Part \textit{%
(i)} is, in essence, a non degeneracy requirement; as far as part \textit{%
(ii)} is concerned, its most immediate consequence is that, under the other
assumptions in this paper, the condition $E\ln \left\vert \varphi
+b_{0}\right\vert <0$ is necessary and sufficient for strict stationarity -
see \citet{aue06}.

\bigskip

When $X_{t}$ is non-stationary, we need the following assumptions in
addition to Assumption \ref{as-1}.

\begin{assumption}
\label{as-3} If $E\ln \left\vert \varphi +b_{0}\right\vert \geq 0$, it holds
that: (i) $e_{0}$ has bounded density; (ii) when $P\left( b_{0}=0\right) <1$%
, $E\left\vert \ln \left\vert \varphi +b_{0}\right\vert \right\vert
^{k}<\infty $ for some $k>2$; (iii) $EX_{0}^{2}<\infty $.
\end{assumption}

\begin{assumption}
\label{as-4}When $E\ln \left\vert \varphi +b_{0}\right\vert =0$ with $%
b_{0}=0 $ a.s., it holds that either (i) $E\left\vert X_{0}\right\vert
^{2}<\infty $ and $E\left\vert e_{0}\right\vert ^{\nu ^{\prime }}<\infty $
for some $\nu ^{\prime }>2$; or (ii) (a) $\{e_{t},-\infty <t<\infty \}$ are
symmetric random variables with common distribution $F\left( x\right) $ such
that%
\begin{equation*}
1-F\left( x\right) =C_{0}x^{-\gamma }+\varsigma \left( x\right) x^{-\gamma },%
\text{ \ \ }x\geq x_{0},
\end{equation*}%
with $C_{0}>0$, $\gamma \in \left( 0,2\right] $ and $\varsigma \left(
x\right) \rightarrow 0$ as $x\rightarrow \infty $, with $\varsigma \left(
x\right) x^{-\gamma }$ decreasing for all $x\geq x_{0}$; and (b) $%
E\left\vert X_{0}\right\vert ^{\gamma }<\infty $.
\end{assumption}

Assumption \ref{as-3} is relatively common in this literature (see e.g. %
\citealp{berkes2009}). The main part of the assumption is part \textit{(ii)}%
, which poses a moment restriction on $\ln \left\vert \varphi
+b_{0}\right\vert $: asymptotics is based on this quantity rather than on $%
\left\vert \varphi +b_{0}\right\vert $; note that this is not, therefore, a
requirement on the existence of e.g. the second moment of $b_{0}$.
Assumption \ref{as-4} deals with the standard unit root case, where $\varphi
=1$ and $b_{0}=0$ a.s.; all other cases are covered by Assumption \ref{as-3}%
. Part \textit{(ii)} of the assumption, in particular, allows for infinite
variance, and indeed only requires minimal moment existence conditions - all
that is needed is that the tail index, $\gamma $, be strictly positive, so
that the variance, and even the first absolute moment, need not be finite.
Note that we do not need to estimate $\gamma $ at any stage of the proposed
testing procedure. The requirements on the tail behaviour of the
distribution function are rather standard in the literature (%
\citealp{berkes1986}; \citealp{berkes1989}). Some of the technical results
derived under this assumption are of general interest, such as the
anti-concentration bound in Lemma \ref{berkes}.

\subsection{Detecting strict stationarity\label{testing}}

We start by discussing the rationale underpinning the construction of the
test statistic. In the light of the comments above, $\left\vert
X_{t}\right\vert \rightarrow \infty $ or not according as $X_{t}$ is
non-stationary or stationary: this holds under quite general circumstances,
e.g. whether $b_{t}=0$ or not, or whether $E\left\vert X_{t}\right\vert
^{2}<\infty $ or not. Thus, it could be possible to exploit this fact to
decide between $H_{0}$ and $H_{A}$. Heuristically, a possible indicator
could be based on 
\begin{equation}
E\left( u_{t}^{2}|I_{t-1}\right) =Ee_{0}^{2}+X_{t-1}^{2}Eb_{0}^{2},
\label{cnd-var}
\end{equation}%
where $u_{t}=X_{t}-\varphi X_{t-1}$\ and $I_{t-1}$\ is the information set
available up to $t-1$\ - (\ref{cnd-var}) represents the (conditional)
variance of the \textquotedblleft error term\textquotedblright\ $u_{t}$.
Testing for stationarity or the lack thereof based on the growth rate of
variances has already been considered (see e.g. \citealp{caishintani}, %
\citealp{bandi2014}, \citealp{corradiflil}). Albeit natural, this approach
suffers from several drawbacks: $Eb_{0}^{2}$ and/or $Ee_{0}^{2}$ may not
exist (thus limiting the applicability of the test); or $Eb_{0}^{2}$ could
be zero, which would prevent the test from being applied to genuine $%
AR\left( 1\right) $ specifications; finally, approaches based on (\ref%
{cnd-var}) may require estimates of $Eb_{0}^{2}$ and $Ee_{0}^{2}$, with the
latter not being always consistent (see \citealp{HT16}). In order to
overcome all these difficulties, one could instead think of using the
transformation%
\begin{equation*}
Y_{t}=\frac{a}{a+X_{t}^{2}},
\end{equation*}%
where $0<a<\infty $ is chosen so as to ensure scale invariance.
Heuristically, since $a>0$, $Y_{t}$ should not be equal to $0$ when $X_{t}$
is stationary; conversely, since $a<\infty $, when $X_{t}$ is nonstationary, 
$Y_{t}$ should drift to $0$. The variable $Y_{t}$ is not affected by $X_{t} $
having infinite variance or infinite mean, since all moments of $Y_{t}$
exist by construction. Also, in the definition of $Y_{t}$ there is no
dependence on $Eb_{0}^{2}$: thus, $Y_{t}$ uses the full force of $X_{t}$
when this diverges, even in the presence of a genuine $AR\left( 1\right) $
specification for which $b_{t}=0$. Finally, upon choosing $a$ in a suitable
way, no estimation of $Eb_{0}^{2}$ or $Ee_{0}^{2}$ is required, making the
problem much more tractable. We build on these considerations in order to
propose a test for the strict stationarity of $X_{t}$, based on the
following scale-invariant transformation%
\begin{equation}
D_{T}=\frac{1}{T-p}\sum_{t=p+1}^{T}\frac{v_{p}}{v_{p}+X_{t}^{2}},
\label{dt-tilde}
\end{equation}%
where $v_{p}=p^{-1}\sum_{t=1}^{p}X_{t}^{2}$ can be viewed as a sample second
moment (or, if using demeaned data, a sample variance). We point out that,
as shown in the remainder of the paper, $v_{p}$ can still be used even when
the second moment of $X_{t}$\ is not finite. Based on the heuristic
considerations above, $D_{T}$ should converge to a strictly positive number
under stationarity, and to zero otherwise. As a final, ancillary comment, we
note that a related quantity to $D_{T}$ is used in the context of the
estimation of $\varphi $ - see \citet{praskova2004}.

\bigskip

The computation of $v_{p}$ should satisfy the following assumption:

\begin{assumption}
\label{p(T)}It holds that $p=p\left( T\right) $ with: (i) $%
\lim_{T\rightarrow \infty }p\left( T\right) =\infty $; (ii) $\lim
\sup_{T\rightarrow \infty }\frac{p\left( T\right) }{\ln \ln T}=C_{0}<\infty $%
.
\end{assumption}

The rates of convergence are summarised in the following theorem.

\begin{theorem}
\label{dt-tilde-convergence}Under Assumptions \ref{as-1}-\ref{p(T)}, it
holds that%
\begin{eqnarray}
D_{T} &=&C_{0}+o\left( 1\right) \text{ when }E\ln \left\vert \varphi
+b_{0}\right\vert <0,  \label{th2.1-a} \\
D_{T} &=&O\left( T^{-\epsilon }\right) \text{ when }E\ln \left\vert \varphi
+b_{0}\right\vert \geq 0,  \label{th2.1-b}
\end{eqnarray}%
where $0<C_{0}<\infty $ and $\epsilon >0$.
\end{theorem}

The constant $C_{0}\leq 1$ in (\ref{th2.1-a}) is explicitly calculated in
the appendix, where it is shown that its value differs according as $%
EX_{0}^{2}<\infty $ or $=\infty $ - in the latter case, $C_{0}=1$. However,
for the purpose of the implementation of the test, it suffices to have $%
C_{0}>0$. The result in (\ref{th2.1-b}) means that, under nonstationarity, $%
D_{T}$ drifts to zero. We do not know anything about the value of $\epsilon $
in general; however, on account of (\ref{th2.1-b}), $D_{T}$ converges a.s.
to zero at a polynomial rate.\newline

Based on Theorem \ref{dt-tilde-convergence}, we can propose a test for $%
H_{0} $. Consider a sequence $\psi \left( T\right) $ such that, as $%
T\rightarrow \infty $%
\begin{equation}
\psi \left( T\right) \rightarrow \infty \text{ \ \ and \ }T^{-\epsilon }\psi
\left( T\right) \rightarrow 0.  \label{psi-2}
\end{equation}%
Then, by virtue of (\ref{th2.1-a}) and (\ref{th2.1-b}), we can assume that%
\begin{eqnarray}
\lim_{T\rightarrow \infty }\psi \left( T\right) D_{T} &=&\infty \text{ under 
}H_{0},  \label{lim-1} \\
\lim_{T\rightarrow \infty }\psi \left( T\right) D_{T} &=&0\text{ under }%
H_{A}.  \label{lim-2}
\end{eqnarray}%
Thus, $\psi \left( T\right) D_{T}$ diverges to positive infinity under the
null, whereas it drifts to zero under the alternative. Equations (\ref{lim-1}%
)-(\ref{lim-2}) suggest that $\psi \left( T\right) D_{T}$ can be a suitable
diagnostic to discriminate between $E\ln \left\vert \varphi
+b_{0}\right\vert <0$ and $E\ln \left\vert \varphi +b_{0}\right\vert \geq 0$.%
\newline

Let now $g\left( \cdot \right) $ be a continuous, monotonically increasing
function such that $g\left( 0\right) =0$ and $\lim_{x\rightarrow \infty
}g\left( x\right) =\infty $, and define%
\begin{equation*}
l_{T}=g\left( \psi \left( T\right) D_{T}\right) ;
\end{equation*}%
based on (\ref{lim-1}) and (\ref{lim-2}), it holds that%
\begin{eqnarray}
&&%
\begin{tabular}{ll}
$P\left\{ \omega :\lim_{T\rightarrow \infty }l_{T}=\infty \right\} =1$ & $%
\text{ under }H_{0},$%
\end{tabular}
\label{lt-null-as} \\
&&%
\begin{tabular}{ll}
$P\left\{ \omega :\lim_{T\rightarrow \infty }l_{T}=0\right\} =1$ & $\text{
under }H_{A},$%
\end{tabular}
\label{lt-alternative-as}
\end{eqnarray}%
and consequently we can assume that $\lim_{T\rightarrow \infty }l_{T}=\infty 
$ under $H_{0}$, and $=0$ under $H_{A}$. Note that, on account of (\ref%
{psi-2}), $\psi \left( T\right) $ may not be allowed to diverge too fast; $%
\psi \left( T\right) =\left( \ln T\right) ^{\beta }$, for some $\beta >0$,
is a possible choice. However, depending on the choice of the function $%
g\left( \cdot \right) $, the sequence $l_{T}$ can be made to diverge
arbitrarily fast.

Our test is based on a randomised version of $l_{T}$. We propose a
\textquotedblleft classical\textquotedblright\ randomisation scheme, which
has been employed in the literature - we refer to \citet{corradi2006} and %
\citet{bandi2014} \textit{inter alia. }Of course, other schemes are also
possible.

\begin{description}
\item[\textbf{Step 1}] Generate an \textit{i.i.d. }sequence $\left\{ \xi
_{j}\right\} $, $1\leq j\leq R$, with common distribution $G\left( \cdot
\right) $.

\item[\textbf{Step 2}] Generate the Bernoulli sequence $\zeta _{j}=I\left(
l_{T}^{1/2}\xi _{j}\leq u\right) $, with $u$ extracted from a distribution $%
F\left( u\right) $.

\item[\textbf{Step 3}] Compute%
\begin{equation}
\vartheta _{R,T}\left( u\right) =\frac{\left[ G\left( 0\right) \left(
1-G\left( 0\right) \right) \right] ^{-1/2}}{R^{1/2}}\sum_{j=1}^{R}\left(
\zeta _{j}-G\left( 0\right) \right) .  \label{theta-minor}
\end{equation}

\item[\textbf{Step 4}] Define%
\begin{equation}
\Theta _{R,T}=\int_{-\infty }^{+\infty }\left\vert \vartheta _{R,T}\left(
u\right) \right\vert ^{2}dF\left( u\right) .  \label{theta-maior}
\end{equation}
\end{description}

We need the following regularity conditions on $G\left( \cdot \right) $ and $%
F\left( \cdot \right) $:

\begin{assumption}
\label{as-5}It holds that: (i) $G\left( \cdot \right) $ has bounded density
and $G\left( 0\right) \neq 0$ or $1$; (ii) $\int_{-\infty }^{+\infty
}u^{2}dF\left( u\right) <\infty $.
\end{assumption}

We are now ready to present the main results. Let $P^{\ast }$ denote the
conditional probability with respect of $\{e_{t},b_{t},-\infty <t<\infty \}$%
; we use the notation \textquotedblleft $\overset{D^{\ast }}{\rightarrow }$%
\textquotedblright\ and \textquotedblleft $\overset{P^{\ast }}{\rightarrow }$%
\textquotedblright\ to define, respectively, conditional convergence in
distribution and in probability according to $P^{\ast }$.

\begin{theorem}
\label{theta}Under Assumptions \ref{as-1}-\ref{as-5}, as $\min \left(
R,T\right) \rightarrow \infty $ with (\ref{psi-2}) and%
\begin{equation}
\frac{R^{1/2}}{g\left( \psi \left( T\right) \right) }\rightarrow 0,
\label{restriction}
\end{equation}%
it holds that%
\begin{equation}
\Theta _{R,T}\overset{D^{\ast }}{\rightarrow }\chi _{1}^{2}\text{ \ \ under }%
H_{0},  \label{asy-null}
\end{equation}%
for almost all realisations of $\left\{ b_{t},e_{t},-\infty <t<\infty
\right\} $.

Under Assumptions \ref{as-1} and \ref{as-3}-\ref{p(T)}, as $\min \left(
R,T\right) \rightarrow \infty $ with (\ref{psi-2}), it holds that%
\begin{equation}
\frac{G\left( 0\right) }{1-G\left( 0\right) }\frac{1}{R}\Theta _{R,T}\overset%
{P^{\ast }}{\rightarrow }\frac{\int_{-\infty }^{+\infty }\left[ I_{\left[
0,\infty \right) }\left( u\right) -G\left( 0\right) \right] ^{2}dF\left(
u\right) }{G\left( 0\right) \left[ 1-G\left( 0\right) \right] }\text{ \ \
under }H_{A},  \label{asy-alternative}
\end{equation}%
for almost all realisations of $\left\{ b_{t},e_{t},-\infty <t<\infty
\right\} $.
\end{theorem}

Theorem \ref{theta} provides the limiting behaviour of the test statistic $%
\Theta _{R,T}$ under the null and under the alternative. The results are
derived conditional on the sample, and they hold for all possible
realisations, apart from a set of measure zero. Considering (\ref%
{asy-alternative}), note that the drift term is maximised for $G\left(
0\right) =F\left( 0\right) =\frac{1}{2}$, which entails that, under $H_{A}$, 
$R^{-1}\Theta _{R,T}\overset{P^{\ast }}{\rightarrow }1$.

The theorem illustrates how the choice of $R$ impacts on the behaviour of $%
\Theta _{R,T}$. On the one hand, one should choose $R$ as large as possible,
in order to maximise the rate of divergence of $\Theta _{R,T}$ under the
alternative - this is evident from (\ref{asy-alternative}). On the other
hand, the test statistic has a non-centrality parameter which grows with $R$%
, under the null: this is illustrated by (\ref{restriction}). Consequently, $%
R$ should not be too large, in order to avoid size distortion.

Theorem \ref{theta} implies that%
\begin{eqnarray}
\lim_{\min (T,R)\rightarrow \infty }P^{\ast }\{\Theta _{R,T} &\geq
&c_{\alpha }\}=\alpha \;\text{under}\;H_{0},  \label{size-gamma} \\
\lim_{\min (T,R)\rightarrow \infty }P^{\ast }\{\Theta _{R,T} &\geq
&c_{\alpha }\}=1\;\text{under}\;H_{A},  \label{power-gamma}
\end{eqnarray}%
for almost all realisations of $\left\{ b_{t},e_{t},-\infty <t<\infty
\right\} $, where $c_{\alpha }$ is defined as $P\{N\left( 0,1\right) \geq
c_{\alpha }\}=\alpha $, $\alpha \in \left( 0,1\right) $.

\subsubsection{Deciding between $H_{0}$\ and $H_{A}\label{deciding}$}

Testing using $\Theta _{R,T}$ is, in essence, based on checking a rate of
convergence, and it is therefore very similar to the idea in %
\citet{bandi2014} - we also refer to a contribution by \citet{kanaya} for a
discussion. In principle (although, possibly, with some interpretational
difficulties related to having a null hypothesis spelt out in terms of a
rate of divergence), the test can be employed as it is, and its main
properties (size and power) are reported in (\ref{size-gamma}) and (\ref%
{power-gamma}). Based on the latter equation, the test rejects the null with
probability one when false, thus being consistent. Conversely, the meaning
of (\ref{size-gamma}) is non-standard. The randomness in $\Theta _{R,T}$ is
added by the researcher, and indeed it is the only randomness present in the
statistic: such randomness does not vanish asymptotically. Thus, different
researchers using the same dataset will obtain different $p$-values; indeed,
if an infinite number of researchers apply the test to the same data (and
the null holds), the resulting $p$-values will be uniformly distributed on $%
\left[ 0,1\right] $.

This is a well-known feature of randomised tests, and it may be viewed as an
undesirable issue, which may explain their relative infrequent use - see the
discussion and the solutions proposed in the contribution by \citet{geyer}.
We build on the notion of \textquotedblleft randomised confidence
function\textquotedblright , proposed by \citet{song}, in order to propose a
strong rule to decide between $H_{0}$ and $H_{A}$, whose outcome is the same
for all researchers using the same dataset. In order to remove the
randomness from the test statistic, each researcher, instead of computing $%
\Theta _{R,T}$ just once, will compute the test statistic $S$ times using,
at each iteration $s$, an independent sequence $\left\{ \xi _{j}^{\left(
s\right) }\right\} $ for $1\leq j\leq R$ and $1\leq s\leq S$, then defining%
\begin{equation}
Q\left( \alpha \right) =S^{-1}\sum_{s=1}^{S}I\left[ \Theta _{R,T}^{\left(
s\right) }\leq c_{\alpha }\right] ;  \label{q}
\end{equation}%
in our context, $Q\left( \alpha \right) $ is the randomised confidence
function proposed by \citet{song}, computed under the null. Intuitively,
based on Theorem \ref{theta}, $Q\left( \alpha \right) $ should converge to $%
1-\alpha $ under $H_{0}$, and to $0$ under $H_{A}$. This dichotomous
behaviour is not subject to the randomness added by the researcher (which is
washed away as $S\rightarrow \infty $), and it could be employed to
construct a decision rule based on the Law of the Iterated Logarithm; in
particular, we propose to decide in favour of $H_{0}$ if%
\begin{equation}
Q\left( \alpha \right) \geq D_{\alpha ,S},  \label{decision}
\end{equation}%
with%
\begin{equation}
D_{\alpha ,S}=\left( 1-\alpha \right) -\sqrt{\alpha \left( 1-\alpha \right) }%
\sqrt{\frac{2\ln \ln S}{S}}.  \label{decision-bound}
\end{equation}%
It holds that

\begin{corollary}
\label{strong-rule}Under the same assumptions as Theorem \ref{theta}, as $%
\min \left( T,R,S\right) \rightarrow \infty $ with (\ref{restriction}), it
holds that%
\begin{equation}
P^{\ast }\left( \lim_{\min (T,R,S)\rightarrow \infty }Q\left( \alpha \right)
\geq D_{\alpha ,S}\right) =1\text{ \ \ under }H_{0},  \label{strong-nul}
\end{equation}%
and%
\begin{equation}
P^{\ast }\left( \lim_{\min (T,R,S)\rightarrow \infty }Q\left( \alpha \right)
=0\right) =1\text{ \ \ under }H_{A},  \label{strong-alt}
\end{equation}%
for almost all realisations of $\left\{ b_{t},e_{t},-\infty <t<\infty
\right\} $ and for every $\alpha >0$.
\end{corollary}

The rule in (\ref{decision}) is similar to the rule proposed by %
\citet{corradiflil}, who proposes a bound to discern between $I\left(
0\right) $ and $I\left( 1\right) $ - related contributions, where criteria
are proposed in the context of choosing between $I\left( 0\right) $ and $%
I\left( 1\right) $, have also been developed by \citet{stock}, %
\citet{ploberger1994} and \citet{ploberger1996}. A typical advantage of this
approach is that, at least asymptotically, it yields a probability zero of
having both a Type I error and a Type II error, as (\ref{strong-nul}) and (%
\ref{strong-alt}) show. In addition to this, as far as our context is
concerned, we note that, based on Corollary \ref{strong-rule}, under $H_{0}$
each researcher will make the same decision (\textquotedblleft
accept\textquotedblright ), with no discrepancies among researchers and
probability one of being correct, and similarly under $H_{A}$.

\section{Discussion and extensions\label{discussion}}

In this section, we discuss possible extensions and generalisations of the
basic test statistic. We consider three possible generalisations: (a) the
case of deterministics being present in (\ref{rca}); (b) the construction of
a test for the null of non-stationarity; and finally, (c) the construction
of different, but related test statistics. All results are shown in appendix
when necessary; the proofs of some results however follow readily from the
existing proofs, and thus, when possible, we omit the details.

\subsection{Employing the test in the presence of deterministics\label%
{deterministics}}

So far, we have assumed no deterministics in (\ref{rca}). This is common in
this literature, when heavy tails are considered, and we also refer to the
comments in \citet{cavaliere2016}. In this section, we show that it is
possible to consider an extension to incorporate deterministics in (\ref{rca}%
). We assume that the observed data - say $X_{t}^{\ast }$ - are generated as%
\begin{equation}
X_{t}^{\ast }=d_{t}+X_{t},  \label{rca-det}
\end{equation}%
where $X_{t}$ is defined as in (\ref{rca}). There are two types of
deterministic processes $d_{t}$. In the case of square integrable $d_{t}$,
viz.%
\begin{equation}
\lim_{T\rightarrow \infty }\frac{1}{T}\sum_{t=1}^{T}\left\vert
d_{t}\right\vert ^{2}<\infty ,  \label{integrable}
\end{equation}%
we show that the test can be used with no modifications. Condition (\ref%
{integrable}) includes several possible cases: $d_{t}$ can be constant; it
can be piecewise constant, thus allowing for shifts in the mean; or it could
be a weighted average of sines and cosines, which could be useful to model
seasonalities and, in general, smooth, bounded processes (see %
\citealp{enders2012}). In all these cases, and indeed whenever (\ref%
{integrable}) is satisfied, the test can be applied directly, with no
modifications or prior knowledge of the nature of $d_{t}$.

It is also possible to apply the test in the presence of trends, and in
general when (\ref{integrable}) does not hold. In this case, it is necessary
to detrend the data first, by estimating $d_{t}$ via, say, $\widehat{d}_{t}$%
. The test statistic $D_{T}$ is modified as%
\begin{equation*}
D_{T}^{\ast }=\frac{1}{T-p}\sum_{t=p+1}^{T}\frac{v_{p}}{v_{p}+Y_{t}^{2}},
\end{equation*}%
where%
\begin{equation*}
Y_{t}=\left\{ 
\begin{tabular}{l}
$X_{t}^{\ast }$ \\ 
$X_{t}^{\ast }-\widehat{d}_{t}$%
\end{tabular}%
\right. \text{ according as 
\begin{tabular}{l}
(\ref{integrable}) holds true \\ 
(\ref{integrable}) does not hold%
\end{tabular}%
.}
\end{equation*}%
We formalise our discussion in the following assumption.

\begin{assumption}
\label{as-8}It holds that (i) either (a) (\ref{integrable}) holds; or (b) it
holds that $E\left\vert \widehat{d}_{t}-d_{t}\right\vert ^{2}=O\left(
T^{-\epsilon _{1}}\right) $ for some $\epsilon _{1}>0$; (ii) when $E\ln
\left\vert \varphi +b_{0}\right\vert <0$, $P\left( |\bar{X}_{0}|=c\right) <1$
for all $c\in R$.
\end{assumption}

Part \textit{(i)}(b) of the assumption is very similar, in spirit, to
Assumption 5 in \citet{kapetanios2007} and it essentially requires that $%
\widehat{d}_{t}$ be a consistent estimator of the (trend) function $d_{t}$.
Although several choices are possible, we refer to the contribution by %
\citet{yao2004} on the estimation of trend functions in the presence of
heavy tails, where MSE consistency is still ensured; note that we do not
need $\epsilon _{1}$ to be any special value, as long as the MSE\ drifts to
zero at a polynomial rate. Part \textit{(ii)} strengthens Assumption \ref%
{as-2}, and it is again a non-degeneracy condition.

It holds that

\begin{theorem}
\label{dt-deterministics}Under Assumptions \ref{as-1}-\ref{p(T)} and \ref%
{as-8}, equations (\ref{th2.1-a}) and (\ref{th2.1-b}) hold.
\end{theorem}

Theorem \ref{dt-deterministics} entails that $D_{T}^{\ast }$ can be used in
the same way as $D_{T}$, obtaining the same results for the corresponding
test (under the same assumptions), and it can be generalised as we do with $%
D_{T}$\ in the next sections.

\subsection{Testing for the null of non-stationarity\label{swapped}}

In the spirit of the KPSS test (see \citealp{kpss}; see also the
contribution by \citealp{giraitis}), and of other constributions in the
context of nonlinear models (see e.g. \citealp{kapetanios2007}), all the
results developed so far are based on the hypothesis testing framework set
out in (\ref{fmwk}), where $H_{0}:$ $X_{t}$ is strictly stationary.

However, a test for 
\begin{equation}
\left\{ 
\begin{tabular}{ll}
$H_{0}^{\ast }:$ & $X_{t}\text{ is nonstationary}$ \\ 
$H_{A}^{\ast }:$ & $X_{t}\text{ is strictly stationary}$%
\end{tabular}%
\right.  \label{test-swap}
\end{equation}%
can be readily derived from the theory developed above.

Our testing approach requires having a test statistic which diverges under
the null, whilst being bounded under the alternative. On account of (\ref%
{lim-1}) and (\ref{lim-2}), one could use%
\begin{equation}
l_{T}^{\ast }=g\left( \frac{1}{\psi \left( T\right) D_{T}}\right) ,
\label{lt-star}
\end{equation}%
where $\psi \left( T\right) $ is defined in (\ref{psi-2}); by continuity, it
follows that%
\begin{equation*}
\begin{tabular}{ll}
$P\left\{ \omega :\lim_{T\rightarrow \infty }l_{T}^{\ast }=\infty \right\}
=1 $ & $\text{under }H_{0}^{\ast },$ \\ 
$P\left\{ \omega :\lim_{T\rightarrow \infty }l_{T}^{\ast }=0\right\} =1$ & $%
\text{under }H_{A}^{\ast },$%
\end{tabular}%
\end{equation*}%
from which a test, based on a randomised version of $l_{T}^{\ast }$, can be
constructed. Using the same algorithm as proposed above, we would obtain a
test statistic denoted by $\Theta _{R,T}^{\ast }$, whose asymptotics is in
the following theorem, reported without proof.

\begin{theorem}
\label{gammaswap}We assume that Assumptions \ref{as-1}-\ref{as-5} are
satisfied. Then, under $H_{0}^{\ast }$, as $\min \left( T,R\right)
\rightarrow \infty $ with 
\begin{equation}
\frac{R^{1/2}}{g\left( \frac{T^{\epsilon }}{\psi \left( T\right) }\right) }%
\rightarrow 0,  \label{restrictions}
\end{equation}%
for every $\epsilon>0$, it holds that $\Theta _{R,T}^{\ast }\overset{D^{\ast
}}{\rightarrow }N\left( 0,1\right) $ for almost all realisations of $\left\{
b_{t},e_{t},-\infty <t<\infty \right\} $. Under $H_{A}^{\ast }$, it holds
that as $\min \left( T,R\right) \rightarrow \infty $%
\begin{equation*}
\frac{G\left( 0\right) }{1-G\left( 0\right) }\frac{1}{R}\Theta _{R,T}^{\ast }%
\overset{P^{\ast }}{\rightarrow }\frac{\int_{-\infty }^{+\infty }\left[ I_{%
\left[ 0,\infty \right) }\left( u\right) -G\left( 0\right) \right]
^{2}dF\left( u\right) }{G\left( 0\right) \left[ 1-G\left( 0\right) \right] },
\end{equation*}%
for almost all realisations of $\left\{ b_{t},e_{t},-\infty <t<\infty
\right\} $.
\end{theorem}

As for $\Theta _{R,T}$, (\ref{restrictions}) provides a selection rule for $%
R $. If, as suggested in the next section, one were to choose $g\left(
x\right) =\exp \left( x\right) -1$ - and $\psi \left( T\right) =\left( \ln
T\right) ^{\beta }$, as recommended above - then setting $R=T$ would satisfy
(\ref{restrictions}). Finally, note that instead of $X_{t}$ the same
arguments could be applied to $Y_{t}$.

\subsection{Modifications of the test statistic\label{modify}}

Heuristically, our test statistic is based on studying the variance of $%
u_{t}=e_{t}+b_{t}X_{t-1}$ in%
\begin{equation*}
X_{t}=\varphi X_{t-1}+u_{t}.
\end{equation*}%
The main intuition is that under stationarity $E\left(
u_{t}^{2}|I_{t-1}\right) $ should be bounded as $t$ elapses, whereas it
grows as $t\rightarrow \infty $ when $X_{t}$ is nonstationary. Owing to the
reasons discussed at the beginning of Section \ref{testing}, the testing
procedure uses a modified statistic, rather than $E\left(
u_{t}^{2}|I_{t-1}\right) $ directly.

In a similar vein, one use different moments of the error term $u_{t}$, say $%
E\left( \left\vert u_{t}\right\vert ^{\varsigma }|I_{t-1}\right) $ for some $%
\varsigma >0$: it can be expected that the stationarity/nonstationarity of $%
X_{t}$ will still entail the convergence/divergence of $E\left( \left\vert
u_{t}\right\vert ^{\varsigma }|I_{t-1}\right) $. From a technical point of
view, our arguments and our proofs differ depending on whether $E\left\vert
X_{t}\right\vert ^{2}=\infty $ or not, and therefore it can be envisaged
that, if relying upon $E\left( \left\vert u_{t}\right\vert ^{\varsigma
}|I_{t-1}\right) $, proofs will differ according as $E\left\vert
X_{t}\right\vert ^{\varsigma }=\infty $ or not - but apart from this, the
final results will be the same.

To formalise the discussion, consider%
\begin{equation}
D_{T}\left( \varsigma \right) =\frac{1}{T-p}\sum_{t=p+1}^{T}\frac{%
v_{p}\left( \varsigma \right) }{v_{p}\left( \varsigma \right) +\left\vert
X_{t}\right\vert ^{\varsigma }},  \label{dt-csi}
\end{equation}%
where $v_{p}\left( \varsigma \right) =p^{-1}\sum_{t=1}^{p}\left\vert
X_{t}\right\vert ^{\varsigma }$. The following result is reported without
proof.

\begin{theorem}
\label{dt-csi-theorem}Under Assumptions \ref{as-1}-\ref{p(T)}, equations (%
\ref{th2.1-a}) and (\ref{th2.1-b}) hold.
\end{theorem}

Theorem \ref{dt-csi-theorem} entails that $D_{T}\left( \varsigma \right) $
has the same properties as $D_{T}$, and it can therefore be used, and
generalised, in the same way. Indeed, upon making sure that the artificial
samples used in the various randomisations are generated independently, it
would even be possible to run a meta-test by trying several values of $%
\varsigma $ and combining the outcomes e.g. according to Fisher's method. As
before, the same ideas could be applied to $Y_{t}$. Although this extension
is theoretically possible, unreported simulations show that $\varsigma =2$
affords the best results; choosing $\varsigma <2$ makes tests very
conservative (e.g. when $\varsigma =1$ the power is cut by two thirds),
whereas $\varsigma >2$ results in the opposite problem.

\section{Simulations and empirical illustration\label{montecarlo}}

This section contains two separate contributions. In Section \ref{mc} we
report some evidence from synthetic data on the empirical rejection
frequencies of our test in order to assess size and power; we analyse the
performance of decision rules based on $D_{\alpha ,S}$; and we discuss
possible guidelines for the implementation of the test statistic. In Section %
\ref{empirics}, we illustrate our approach, and in particular the use of $%
D_{\alpha ,S}$, through an application to several US\ macro aggregates.

\subsection{Monte Carlo evidence\label{mc}}

The design of the reported experiments is as follows.

We use (\ref{rca}) as a DGP; $b_{t}$ is generated as \textit{i.i.d.} $%
N\left( 0,\sigma _{b}^{2}\right) $, with $\sigma _{b}^{2}\in \left\{
0,0.1,0.25\right\} $ in order to consider the genuine AR case as well as
cases with random coefficients. According to the theory, it would also be
possible to consider a heavy tailed distribution for $b_{t}$; we noticed
through few trials that doing this does not change results in a decisive
way. We report results for three different specifications of $e_{t}$: 
\textit{i.i.d.} $N\left( 0,\sigma _{e}^{2}\right) $, \textit{i.i.d.} $t_{2}$
and \textit{i.i.d.} $t_{1}$, where $t_{k}$ denotes a Student t distribution
with $k$ degrees of freedom, so as to consider the cases of infinite
variance and infinite mean. In the Gaussian case, we have used $\sigma
_{e}^{2}=1$; we note however that the test is completely insensitive to the
value of $\sigma _{e}^{2}$, which suggests that the use of $v_{p}$ is very
effective at ensuring scale invariance.

We have used $\varphi \in \left\{ 0,0.5,0.75,0.95,1,1.05\right\} $; larger
values of $\varphi $, for the nonstationary cases, could also be considered
but in those cases - as can be expected - the test has unit power even for
very small samples.

\bigskip

The various combinations $\left\{ \varphi ,\sigma _{b}^{2}\right\} $ deserve
attention. All cases where $\varphi \leq 0.95$ entail that $X_{t}$ is
stationary: the corresponding empirical rejection frequencies represent the
size of the test. Also, upon computing the value of $E\ln \left\vert \varphi
+b_{0}\right\vert $, it can be noted that the two cases $\left\{ \varphi
,\sigma _{b}^{2}\right\} =\left\{ 1,0.1\right\} $ and $\left\{ \varphi
,\sigma _{b}^{2}\right\} =\left\{ 1,0.25\right\} $ correspond to a
stationary STUR; even in these cases the empirical rejection frequencies
represent the size, and it should be noted that, when $X_{t}$ is a
(stationary) STUR process, it has infinite variance irrespective of the
distributions of $b_{t}$ and $e_{t}$. Finally, again upon computing $E\ln
\left\vert \varphi +b_{0}\right\vert $, it turns out that the case $\left\{
\varphi ,\sigma _{b}^{2}\right\} =\left\{ 1.05,0.25\right\} $ is also an
instance of $X_{t}$ being stationary, again with infinite variance.

Thus, the nonstationary cases considered in our experiment are a pure
explosive case corresponding to $\left\{ \varphi ,\sigma _{b}^{2}\right\}
=\left\{ 1.05,0\right\} $, and the pure unit root case $\left\{ \varphi
,\sigma _{b}^{2}\right\} =\left\{ 1,0\right\} $. In the latter case, clearly 
$E\ln \left\vert \varphi +b_{0}\right\vert =0$ and therefore $X_{t}$ is on
the cusp between explosive and strictly stationary behaviour - note that we
are considering, by virtue of the several possible distributions of $e_{t}$,
also cases of random walk with infinite variance and mean, in a similar
spirit to \citet{cavaliere2016}. However, in our case the null is
stationarity, not unit root, and therefore the empirical rejection
frequencies represent the power of our test. Finally, we point out that the
case $\left\{ \varphi ,\sigma _{b}^{2}\right\} =\left\{ 1.05,0.1\right\} $
is of particular interest because $E\ln \left\vert \varphi +b_{0}\right\vert
=3.3\times 10^{-3}$ - that is, it is positive (and, therefore, $X_{t}$ is
nonstationary) but very small. Finally, in a separate experiment (the
outcomes are in Tables 2 and 4), we have considered several combinations of $%
\left\{ \varphi ,\sigma _{b}^{2}\right\} $ for which $E\ln \left\vert
\varphi +b_{0}\right\vert =0$, so as to evaluate the behaviour of the test
in those cases.

\bigskip

We now turn to describing the specifications of the test; as a general note,
their impact vanishes for large samples. Our reported experiments are based
on the following choices, which delivered the best results and are thus
recommended as guidelines to the applied user. We choose: 
\begin{eqnarray}
\psi \left( T\right) &=&\left( \ln T\right) ^{5/4},  \label{psi} \\
g\left( x\right) &=&\exp \left( \exp \left( x\right) -1\right) -1;
\label{gx}
\end{eqnarray}%
the choices in (\ref{psi}) and (\ref{gx}) are designed in order to ensure
that $g\left( \psi \left( T\right) D_{T}\right) \asymp \exp \left( T\right) $
under $H_{0}$, and that $g\left( \psi \left( T\right) D_{T}\right) $ drifts
to zero as $T\rightarrow \infty $ under $H_{A}$. Thus, the double
exponential in (\ref{gx}) serves the purpose of \textquotedblleft
divaricating\textquotedblright\ as much as possible the case where $\psi
\left( T\right) D_{T}$ diverges from the case where it does not; other
choices would also be possible, but (\ref{psi}) and (\ref{gx}) work well in
all cases considered. Based on (\ref{restriction}), we set $R=T$. Finally,
by Assumption \ref{p(T)}, in the computation of $v_{p}$ (which we carry out
with demeaned data), we need to choose $p=C_{0}\ln \ln T$ for some $C_{0}$,
which implies that $p$ does not vary too much as $T$ increases. We have used 
$C_{0}=2$, rounding $C_{0}\ln \ln T$ to the nearest, largest integer;
varying $p$ around this number does not affect the results anyway. Finally,
we have implemented the decision rule based on $D_{\alpha ,S}$ using $S=1000 
$; we note that increasing this number results in better outcomes, at the
(obvious) cost of a higher computational burden. Under each scenario, we
compute the percentage of times that the decision rule is in favour of $%
H_{0} $, using this as a measure of performance.

\bigskip

We generate $\left\{ \xi _{j}\right\} _{j=1}^{R}$\ as \textit{i.i.d.} $%
N\left( 0,1\right) $, and $u$ is $\left\{ -\sqrt{2},\sqrt{2}\right\} $ with
equal probability. Finally, the sample sizes have been chosen as $T\in
\left\{ 250,500,1000,2000\right\} $; the first $1000$ observations have been
discarded to avoid dependence on initial conditions. The number of
replications is set equal to $2000$; when evaluating the size, this entails
that empirical rejection frequencies have a confidence interval of $\left[
0.04,0.06\right] $.

\bigskip

\begin{center}
\textbf{[Insert Tables 1 and 2 somewhere here]}

\bigskip
\end{center}

Table 1 contains the empirical rejection frequencies for the test for $%
H_{0}:X_{t}$ is strictly stationary, using the combinations of $\left\{
\varphi ,\sigma _{b}^{2}\right\} $ indicated above and, in brackets, the
percentage of times that the decision rule based on $D_{\alpha ,S}$ leads to
accepting $H_{0}$. As can be noted, the test has the correct size for almost
all cases considered - exceptions are cases on the boundary (such as $%
\left\{ \varphi ,\sigma _{b}^{2}\right\} =\left\{ 1,0.1\right\} $, or $%
\left\{ \varphi ,\sigma _{b}^{2}\right\} =\left\{ 1.05,0.25\right\} $), but
even in these cases the size becomes correct as $T$ increases. The
distribution of the error term $e_{t}$ does not affect, in general, the
values of the empirical rejection frequencies, with few exceptions. As far
as power is concerned, in the pure unit root case - viz. when $\left\{
\varphi ,\sigma _{b}^{2}\right\} =\left\{ 1,0\right\} $ - the test exhibits
good power, which is found to be higher than $50\%$ whenever $T\geq 500$,
despite not being designed explicitly for that specific alternative
hypothesis; even in this case, the results are broadly similar for different
distributions of $e_{t}$. As a conclusion, the test seems to work well in
discerning between a genuinely unit root process, and a (stationary) STUR
process. The test is also powerful in the purely explosive case $\left\{
\varphi ,\sigma _{b}^{2}\right\} =\left\{ 1.05,0\right\} $, and has some
power also versus the \textquotedblleft boundary\textquotedblright\ case $%
\left\{ \varphi ,\sigma _{b}^{2}\right\} =\left\{ 1.05,0.1\right\} $; in
this case, the power is affected by the distribution of the error term for
small $T$, and it declines as the tails of the distribution of $e_{t}$
become heavier, but this seems to vanish as $T$ increases. Similar
considerations hold for $D_{\alpha ,S}$; note that, when data have heavy
tails, in the case $\left\{ \varphi ,\sigma _{b}^{2}\right\} =\left\{
0.95,0\right\} $ the procedure requires, in order to work sufficiently well, 
$T\geq 1000$.

As mentioned above, we have also considered a broader set of cases where $%
X_{t}$ is nonstationary, in which $E\ln \left\vert \varphi +b_{0}\right\vert
=0$. The power of our test versus these alternatives is in Table 2; the test
has very good power in all cases considered, the only possible exception
being the case with normally distributed errors and $T=250$, but even in
that case the power picks up for larger $T$. Note the major increase in
power when the error term $e_{t}$ has a Student t distribution; this does
not seem to be sensitive to the degrees of freedom of the distribution. Note
also the excellent performance of $D_{\alpha ,S}$ for large $T$.

\bigskip

We have also considered, based on the discussion in Section \ref{swapped},
testing for $H_{0}:X_{t}$ is nonstationary. For brevity, we did not
experiment with $D_{\alpha ,S}$; otherwise, the design of the simulations,
and the specification of the test statistic, are carried out exactly as in
the previous case. As suggested in Section \ref{swapped}, we use%
\begin{equation*}
l_{T}^{\ast }=g\left( \frac{1}{\psi \left( T\right) D_{T}}\right) ,
\end{equation*}%
in the construction of the test.

\begin{center}
\bigskip

\textbf{[Insert Tables 3 and 4 somewhere here]}

\bigskip
\end{center}

The test has the correct size in all cases considered. The power versus
stationarity is strong when $X_{t}$ is \textquotedblleft very
stationary\textquotedblright\ - i.e. when $\varphi =0$ or $0.5$, and it is
anyway above $50\%$ in the \textquotedblleft less
stationary\textquotedblright\ case of having $\varphi =0.95$ when $T\geq 500$%
. Similarly, Table 4 shows that the test has the correct size even in the
nonstationary, but boundary, case $E\ln \left\vert \varphi +b_{0}\right\vert
=0$. We note that, in all cases considered, the distribution of $e_{t}$ does
not seem to play a role on the final results.

\subsection{Empirical illustration\label{empirics}}

The purpose of this section is primarily to illustrate the use of $Q\left(
\alpha \right) $ and of the decision rule based on $D_{\alpha ,S}$. We apply
our procedure to several US macroeconomic aggregates (similarly to %
\citealp{hillpeng2014}). We consider the logs of: real GDP, M2 (as a measure
of the aggregate money supply), CPI (and we also consider inflation, defined
as the log-difference of CPI), and the Industrial Production index. We also
apply our methodology to the (untransformed) rate of unemployment. Finally,
we also consider the 3-months T-bill, inspired by the contribution by %
\citet{rahbek}.

\bigskip

The decision rule (\ref{decision}) is applied in order to choose between%
\begin{equation*}
\left\{ 
\begin{tabular}{ll}
$H_{0}:$ & $X_{t}\text{ is strictly stationary}$ \\ 
$H_{A}:$ & $X_{t}\text{ is nonstationary}$%
\end{tabular}%
\right.
\end{equation*}%
As a further illustration, we apply the test to first-differenced data, in
the cases where a series is found to be nonstationary. Finally, we also use (%
\ref{decision}) to decide between%
\begin{equation*}
\left\{ 
\begin{tabular}{ll}
$H_{0}:$ & $X_{t}\text{ is nonstationary}$ \\ 
$H_{A}:$ & $X_{t}\text{ is strictly stationary}$%
\end{tabular}%
,\right.
\end{equation*}

again considering data in levels and (if need be) in first difference. As
far as the implementation is concerned, a note on deterministics is in
order. We know from Section \ref{deterministics} that our test can, in
general, be applied to non-zero mean data; the test can also be applied in
the presence of trends (albeit only after detrending), which should make our
procedure particularly suitable for macroeconomic aggregates. There has been
much debate on the presence (or absence) of a linear trend in macroeconomic
aggregates, and, in general, as to whether macroeconomic series are better
characterised as having a linear trend or a unit root (the so-called
\textquotedblleft uncertain unit root\textquotedblright ). Starting at least
from the seminal paper by \citet{nelson}, various contributions have
questioned whether such series ought to be modelled as having a unit root or
a linear trend. GDP is a prime example of this debate, and it has been the
subject of several studies: we refer to the classical paper by %
\citet{rudebusch}, and also to \citet{murray} and the extensive literature
review therein. Similarly, some studies seem to suggest that trends may be
present in the CPI (see \citealp{beechey}), and that money aggregates also
may have trends (\citealp{brand}). Whilst the empirical exercise in this
paper is not aimed at addressing the \textquotedblleft uncertain unit
root\textquotedblright\ debate in a comprehensive way, we have taken this
literature into account by detrending all the series using the GLS
detrending scheme proposed in \citet{elliot}.

\begin{center}
\bigskip

\textbf{[Insert Table 5 somewhere here]}

\bigskip
\end{center}

In the computation of (\ref{decision}), we have used $Q\left( 0.05\right) $,
setting $S=5000$. We have used the same specifications as described in the
previous section, namely: $R=T$; $p=5$; $\psi \left( T\right) =\left( \ln
T\right) ^{5/4}$; $g\left( x\right) =\exp \left( \exp \left( x\right)
-1\right) -1$; and we compute $v_{p}$ using demeaned data. In Table 6, we
also report the estimated values of $\varphi $ and $\sigma _{b}^{2}$
computed using the WLS estimator studied in \citet{HT16}. Based on (\ref%
{decision}), and on the fact that $\alpha =0.05$ and $S=5000$, the decision
rule is based on not rejecting $H_{0}$ whenever%
\begin{equation}
Q\left( 0.05\right) \geq 0.9436,  \label{strong2}
\end{equation}%
rejecting otherwise. Results are reported in Table 6, where we have also
reported, for illustration purposes, the outcomes of the unit root test by %
\citet{elliot} and of the KPSS test (see \citealp{kpss}), which we have
carried out for those series which do not have a random autoregressive root.

\bigskip

\begin{center}
\textbf{[Insert Table 6 somewhere here]}

\bigskip
\end{center}

As a preliminary comment, based on the test for no randomness ($H_{0}:\sigma
_{b}^{2}=0$) developed by \citet{HT16}, two series (Industrial Production
and unemployment) are found to have a random autoregressive root, whereas
the others do not. We have (very) heuristically checked these two series, by
calculating the value of $E\ln \left\vert \varphi +b_{0}\right\vert $, using 
$\widehat{\varphi }$ as face value and assuming that $b_{0}$ is Gaussian; in
both cases, $E\ln \left\vert \varphi +b_{0}\right\vert $ turns out to be
very close to zero. As far as testing for stationarity is concerned, results
are quite clear-cut: all series are found to be nonstationary. This is
perfectly in line with the findings from applying the other unit root tests
reported in the table - although of course these can only be applied to the
series with a deterministic autoregressive root. Interestingly, all series
become stationary after first differencing: this cannot be taken for
granted, since some series have a random autoregressive coefficient and
therefore there is no guarantee that first differencing may induce
stationarity - see \citet{leybourne1996}. Exactly the same pattern of
results is found when swapping the null and the alternative hypotheses.

Finally, we note that, as a robustness check, we have experimented with
different specifications for our procedure - e.g. varying $R$ from $T/2$ to $%
2T$; increasing $p$ to $4\ln \ln T$; and trying $S=1000$ and $S=10000$. In
all these cases, we noted that results were entirely unchanged compared to
the ones in Table 6, suggesting that the strong rule based on (\ref{decision}%
) is quite robust to different choices of user-defined parameters.

\section{Conclusions\label{conclusions}}

In this paper, we have developed a test for the null of strict stationarity
applied to a RCAR(1) model. Our testing approach can be applied to a wide
variety of situations, without requiring any modification or any prior
knowledge. Chiefly, the test is still usable if the autoregressive root is
not random, i.e. in the case of an AR(1) specification. Also, the test does
not require (neither as assumptions, nor for the purpose of the actual
implementation) the existence of the variance of $X_{t}$, or of virtually
any moment. Finally, the test can be applied in the presence of
deterministics: even in this case (with the exception of having trends), the
implementation of the test does not require any prior analysis. To the best
of our knowledge, no existing test has such a level of generality.

Technically, the test is based on the (almost sure) limiting behaviour of a
statistic which either diverges to positive infinity or drifts to zero
without having any randomness, we propose to use it as part of a
randomisation procedure. Numerical evidence shows that the test performs
very well, also showing very promising results in the cases where $X_{t}$ is
borderline between being stationary and non-stationary - that is, in cases
where $E\ln \left\vert \varphi +b_{0}\right\vert $ is either positive or
negative, but \textquotedblleft small\textquotedblright .

\section{Technical Lemmas\label{lemmas}}

The first few lemmas are for the case $E\ln \left\vert \varphi
+b_{0}\right\vert <0$. The first lemma is an immediate consequence of
Assumption \ref{as-2}, and we therefore report it without proof.

\begin{lemma}
\label{positive}Under Assumption \ref{as-2}, if $E\ln \left\vert \varphi
+b_{0}\right\vert <0$ with $\left\vert \varphi \right\vert <1$, it holds
that $E\left\vert \overline{X}_{0}\right\vert ^{\delta }>0$, for all $\delta
\geq 0$.
\end{lemma}

\begin{lemma}
\label{vp_ergodic}Under Assumption \ref{as-1}, if $E\ln \left\vert \varphi
+b_{0}\right\vert <0$ it holds that%
\begin{equation*}
\sum_{t=1}^{T}\left\vert \left\vert X_{t}\right\vert ^{\kappa }-\left\vert 
\overline{X}_{t}\right\vert ^{\kappa }\right\vert =O\left( 1\right) ,
\end{equation*}%
for any $\kappa >0$.
\end{lemma}

\begin{proof}
Using Lipschitz and H\"{o}lder continuity, we have%
\begin{equation*}
\left\vert \left\vert X_{t}\right\vert ^{\kappa }-\left\vert \overline{X}%
_{t}\right\vert ^{\kappa }\right\vert \leq \left\{ 
\begin{array}{c}
C_{0}\left( \left\vert \overline{X}_{t}\right\vert ^{\kappa -1}+\left\vert
X_{t}\right\vert ^{\kappa -1}\right) \left\vert X_{t}-\overline{X}%
_{t}\right\vert  \\ 
C_{0}\left\vert X_{t}-\overline{X}_{t}\right\vert ^{\kappa }%
\end{array}%
\right. \text{ according as }%
\begin{array}{c}
\kappa \geq 1 \\ 
\kappa <1%
\end{array}%
\text{. }
\end{equation*}%
\citet{HT16} show that $\left\vert X_{t}-\overline{X}_{t}\right\vert
=O\left( e^{-C_{0}t}\right) $ for some $C_{0}>0$. Also, by Lemma 2 in %
\citet{aue06}, there exists a $\delta ^{\prime }>0$ such that $E\left\vert 
\overline{X}_{t}\right\vert ^{\delta ^{\prime }}<\infty $. Hence, by the
Borel-Cantelli Lemma (see e.g.\ \citealp{chow2012}, Corollary 3 on p. 90), $%
\left\vert \overline{X}_{t}\right\vert =O\left( \left\vert t\right\vert
^{1/\delta ^{\prime }}\left( \ln t\right) ^{\left( 1+\epsilon \right)
/\delta ^{\prime }}\right) $. The desired result now follows immediately by
putting everything together. 
\end{proof}

We now distinguish the cases $E\overline{X}_{0}^{2}=\infty $ and $E\overline{%
X}_{0}^{2}<\infty $. In the former case, the following lemmas are needed.

\begin{lemma}
\label{vp_nonergodic}Under Assumptions \ref{as-1} and \ref{as-2}, with $E\ln
\left\vert \varphi +b_{0}\right\vert <0$ and $\varphi ^{2}+Eb_{0}^{2}\geq 1$%
,\ it holds that, as $T\rightarrow \infty $, $v_{p}\rightarrow \infty $ a.s.
\end{lemma}

\begin{proof}
Let $\overline{v}_{p}=p^{-1}\sum_{t=1}^{p}\overline{X}_{t}^{2}$. By\ Lemma %
\ref{vp_ergodic}, it holds that $v_{p}-\overline{v}_{p}=O\left(
p^{-1}\right) $. Further, under the assumptions of the lemma, it holds that $%
E\overline{X}_{0}^{2}=\infty $ (\citealp{quinn1982}). Let now $\overline{v}%
_{p}\left( C_{0}\right) =p^{-1}\sum_{t=1}^{p}\min \left\{ \overline{X}%
_{t}^{2},C_{0}^{2}\right\} $, where $C_{0}>0$. By construction, $\overline{v}%
_{p}=\lim \sup_{C_{0}\rightarrow \infty }\overline{v}_{p}\left( C_{0}\right) 
$ $=$ $\lim \sup_{C_{0}\rightarrow \infty }E\min \left\{ \overline{X}%
_{0}^{2},C_{0}^{2}\right\} $ a.s., where the last passage follows from the
ergodic theorem. Also, $\lim \sup_{C_{0}\rightarrow \infty }E\min \left\{ 
\overline{X}_{0}^{2},C_{0}^{2}\right\} $ $=$ $\lim_{C_{0}\rightarrow \infty
}E\min \left\{ \overline{X}_{0}^{2},C_{0}^{2}\right\} $ $=$ $E\overline{X}%
_{0}^{2}$, by monotone convergence, which proves the lemma. 
\end{proof}

\begin{lemma}
\label{limit-one}Under Assumptions \ref{as-1} and \ref{as-2}, with $E\ln
\left\vert \varphi +b_{0}\right\vert <0$ and $\varphi ^{2}+Eb_{0}^{2}\geq 1$%
,\ it holds that, as $T\rightarrow \infty $, $D_{T}\rightarrow 1$ a.s.
\end{lemma}

\begin{proof}
We begin by showing that $v_{p}>0$ a.s.; indeed, by the generalised mean
inequality $v_{p}\geq \left( p^{-1}\sum_{t=1}^{p}X_{t}^{\delta }\right)
^{2/\delta }$; using Lemma \ref{vp_ergodic} and the ergodic theorem, we get
that $v_{p}\geq \left( E\overline{X}_{0}^{\delta }\right) ^{2/\delta }$,
which is strictly positive by Lemma \ref{positive} for any $\delta \in
\left( 0,2\right) $. 

Note now that%
\begin{equation}
\sum_{t=p+1}^{T}\left\vert \frac{v_{p}}{v_{p}+X_{t}^{2}}-\frac{v_{p}}{v_{p}+%
\overline{X}_{t}^{2}}\right\vert \leq \sum_{t=p+1}^{T}\frac{\left\vert
X_{t}^{2}-\overline{X}_{t}^{2}\right\vert }{v_{p}+\overline{X}_{t}^{2}}\leq
C_{0}\sum_{t=p+1}^{T}\left\vert X_{t}^{2}-\overline{X}_{t}^{2}\right\vert
=O\left( 1\right) ;  \label{ergodic-3}
\end{equation}%
the last passage follows from Lemma \ref{vp_ergodic}. Combining (\ref%
{ergodic-3}) with Lemma \ref{vp_nonergodic}, this entails that for every $%
C_{0}$, there is a $p_{0}$ such that for all $p\geq p_{0}$%
\begin{equation*}
\frac{1}{T-p}\sum_{t=p+1}^{T}\frac{\overline{X}_{t}^{2}}{v_{p}+\overline{X}%
_{t}^{2}}\leq \frac{1}{T-p}\sum_{t=p+1}^{T}\frac{\overline{X}_{t}^{2}}{C_{0}+%
\overline{X}_{t}^{2}}\text{ a.s.}
\end{equation*}%
By virtue of the ergodic theorem, this entails that%
\begin{equation*}
\frac{1}{T-p}\sum_{t=p+1}^{T}\frac{\overline{X}_{t}^{2}}{v_{p}+\overline{X}%
_{t}^{2}}\leq E\left( \frac{\overline{X}_{0}^{2}}{C_{0}+\overline{X}_{0}^{2}}%
\right) \text{ a.s.}
\end{equation*}%
Lemma 2 in \citet{aue06} entails that there exists a $\delta ^{\prime }>0$
(with, clearly, $\delta ^{\prime }<2$) such that $E\left\vert \overline{X}%
_{0}\right\vert ^{\delta ^{\prime }}<\infty $, so that%
\begin{equation*}
E\left( \frac{\overline{X}_{0}^{2}}{C_{0}+\overline{X}_{0}^{2}}\right) \leq
E\left( \frac{\overline{X}_{0}^{2}}{C_{0}+\overline{X}_{0}^{2}}\right)
^{\delta ^{\prime }/2}\leq E\left( \frac{\overline{X}_{0}^{2}}{C_{0}}\right)
^{\delta ^{\prime }/2}\leq C_{1}C_{0}^{-\delta ^{\prime }/2},
\end{equation*}%
where $C_{1}<\infty $.\ Thus, for any $\epsilon >0$ there exists a random $%
p_{0}$ such that for all $p\geq p_{0}$%
\begin{equation*}
\frac{1}{T-p}\sum_{t=p+1}^{T}\frac{X_{t}^{2}}{v_{p}+X_{t}^{2}}\leq \epsilon 
\text{ a.s.;}
\end{equation*}%
the desired result follows immediately. 
\end{proof}

Lemmas \ref{vp_nonergodic} and \ref{limit-one} allow to study the behaviour
of $D_{T}$ under the null when $E\overline{X}_{0}^{2}=\infty $. When $E%
\overline{X}_{0}^{2}<\infty $, the following result holds.

\begin{lemma}
\label{ergodic}Under Assumptions \ref{as-1} and \ref{as-2}, if $E\ln
\left\vert \varphi +b_{0}\right\vert <0$ with $Ee_{0}^{2}<\infty $ and $%
\varphi ^{2}+Eb_{0}^{2}<1$, it holds that%
\begin{equation*}
\frac{1}{T-p}\sum_{t=p+1}^{T}\frac{X_{t}^{2}}{v_{p}+X_{t}^{2}}=E\left( \frac{%
\overline{X}_{0}^{2}}{E\overline{X}_{0}^{2}+\overline{X}_{0}^{2}}\right)
+o\left( 1\right) .
\end{equation*}
\end{lemma}

\begin{proof}
Under the conditions of the lemma, it holds that $E\overline{X}%
_{0}^{2}<\infty $ (see Lemma 3 in \citealp{aue06}). We have%
\begin{equation*}
X_{t}^{2}\frac{\left\vert E\overline{X}_{0}^{2}-v_{p}\right\vert }{\left( E%
\overline{X}_{0}^{2}+X_{t}^{2}\right) \left( v_{p}+X_{t}^{2}\right) }\leq 
\frac{\left\vert E\overline{X}_{0}^{2}-v_{p}\right\vert }{E\overline{X}%
_{0}^{2}}.
\end{equation*}%
By virtue of Lemma \ref{positive}, we further have that $E\overline{X}%
_{0}^{2}>0$. By Lemma \ref{vp_ergodic}, we have $v_{p}-p^{-1}\sum_{t=1}^{p}%
\overline{X}_{t}^{2}$ $=$ $O\left( p^{-1}\right) $; also, the ergodic
theorem entails that $p^{-1}\sum_{t=1}^{p}\overline{X}_{t}^{2}-E\overline{X}%
_{0}^{2}=o\left( 1\right) $, so that putting everything together we have $E%
\overline{X}_{0}^{2}-v_{p}=o\left( 1\right) $. Hence%
\begin{equation*}
\frac{1}{T-p}\sum_{t=p+1}^{T}\frac{X_{t}^{2}}{v_{p}+X_{t}^{2}}-\frac{1}{T-p}%
\sum_{t=p+1}^{T}\frac{X_{t}^{2}}{E\overline{X}_{0}^{2}+X_{t}^{2}}=o\left(
1\right) .
\end{equation*}%
Note now that%
\begin{equation}
\sum_{t=p+1}^{T}\left\vert \frac{X_{t}^{2}}{E\overline{X}_{0}^{2}+X_{t}^{2}}-%
\frac{\overline{X}_{t}^{2}}{E\overline{X}_{0}^{2}+\overline{X}_{t}^{2}}%
\right\vert \leq C_{0}\sum_{t=p+1}^{T}\left\vert X_{t}-\overline{X}%
_{t}\right\vert ,  \label{ergodic-1}
\end{equation}%
for some $0<C_{0}<\infty $; thus, by Lemma \ref{vp_ergodic}%
\begin{equation*}
\frac{1}{T-p}\sum_{t=p+1}^{T}\frac{X_{t}^{2}}{E\overline{X}_{0}^{2}+X_{t}^{2}%
}=\frac{1}{T-p}\sum_{t=p+1}^{T}\frac{\overline{X}_{t}^{2}}{E\overline{X}%
_{0}^{2}+\overline{X}_{t}^{2}}+O\left( T^{-1}\right) .
\end{equation*}%
The final result obtains from the ergodic theorem. 
\end{proof}

We now report a series of lemmas to study the nonstationary case, $E\ln
\left\vert \varphi +b_{0}\right\vert \geq 0$.

\begin{lemma}
\label{expect-vp-ha}Under Assumptions \ref{as-1}, \ref{as-3} and \ref{as-4},
it holds that%
\begin{equation*}
E\left\vert X_{t}\right\vert ^{\nu }=%
\begin{array}{c}
O\left( t^{\max \left\{ 0,\nu -1\right\} }\exp \left( C_{0}t\right) \right)
\\ 
O\left( t^{\max \left\{ 0,\nu -1\right\} }\exp \left( C_{0}t\right) \right)
\\ 
O\left( t^{1+\max \left\{ 0,\nu -1\right\} }\right) \text{ \ \ \ \ \ \ \ }%
\end{array}%
\text{ according as }%
\begin{array}{c}
E\ln \left\vert \varphi +b_{0}\right\vert \geq 0\text{ with\ }P\left(
b_{0}=0\right) <1\text{\ } \\ 
E\ln \left\vert \varphi +b_{0}\right\vert >0\text{ with\ }P\left(
b_{0}=0\right) =1 \\ 
E\ln \left\vert \varphi +b_{0}\right\vert =0\text{ with\ }P\left(
b_{0}=0\right) =1%
\end{array}%
,
\end{equation*}%
for some $0<C_{0}<\infty $; $\nu $ is defined in Assumption \ref{as-1}, and,
when Assumption \ref{as-4}\textit{(ii) holds, it is} chosen so that $\nu
<\gamma $.
\end{lemma}

\begin{proof}
Consider the recursive solution%
\begin{equation}
X_{t}=X_{0}\prod\limits_{s=1}^{t}\left( \varphi +b_{s}\right)
+\sum_{s=1}^{t}e_{s}\prod\limits_{z=s}^{t-1}\left( \varphi +b_{z+1}\right) ,
\label{recursive}
\end{equation}%
and note that%
\begin{eqnarray*}
X_{t}^{\nu } &\leq &C_{0}\left[ X_{0}^{\nu }\prod\limits_{s=1}^{t}\left(
\varphi +b_{s}\right) ^{\nu }+\left\vert
\sum_{s=1}^{t}e_{s}\prod\limits_{z=s}^{t-1}\left( \varphi +b_{z+1}\right)
\right\vert ^{\nu }\right] \\
&\leq &C_{0}\left[ X_{0}^{\nu }\prod\limits_{s=1}^{t}\left( \varphi
+b_{s}\right) ^{\nu }+t^{\max \left\{ 0,\nu -1\right\}
}\sum_{s=1}^{t}\left\vert e_{s}\right\vert ^{\nu
}\prod\limits_{z=s}^{t-1}\left\vert \varphi +b_{z+1}\right\vert ^{\nu }%
\right] ,
\end{eqnarray*}%
where $\nu $ is defined in Assumption \ref{as-1} and $C_{0}=1$ when $\nu
\leq 1$ and $2^{\nu -1}$ when $\nu >1$.

We begin by studying the case $P\left( b_{0}=0\right) <1$. By Assumption \ref%
{as-1}, it holds that%
\begin{eqnarray}
&&E\left\vert X_{0}\right\vert ^{\nu }\prod\limits_{s=1}^{t}E\left\vert
\varphi +b_{s}\right\vert ^{\nu }+t^{\max \left\{ 0,\nu -1\right\}
}\sum_{s=1}^{t}E\left\vert e_{s}\right\vert ^{\nu
}\prod\limits_{z=s}^{t-1}E\left\vert \varphi +b_{z+1}\right\vert ^{\nu } 
\notag \\
&\leq &C_{0}\left( E\left\vert \varphi +b_{0}\right\vert ^{\nu }\right)
^{t}\left( 1+t^{\max \left\{ 0,\nu -1\right\} }\sum_{s=1}^{t}\left(
E\left\vert \varphi +b_{0}\right\vert ^{\nu }\right) ^{-s}\right) .
\label{expectation-xsquared}
\end{eqnarray}%
By Jensen's inequality, $\ln E\left\vert \varphi +b_{0}\right\vert ^{\nu
}\geq E\ln \left\vert \varphi +b_{0}\right\vert ^{\nu }\geq 0$, so that $%
1\leq E\left\vert \varphi +b_{0}\right\vert ^{\nu }<\infty $. Thus, $\left(
E\left\vert \varphi +b_{0}\right\vert ^{\nu }\right) ^{t}$ is at most of
order $e^{C_{0}t}$ for some $C_{0}>0$. When $P\left( b_{0}=0\right) =1$ and $%
\left\vert \varphi \right\vert >1$, the same result readily follows.
Finally, when $\left\vert \varphi \right\vert =1$ and $P\left(
b_{0}=0\right) =1$, (\ref{expectation-xsquared}) boils down to%
\begin{equation*}
E\left\vert X_{0}\right\vert ^{\nu }\prod\limits_{s=1}^{t}E\left\vert
\varphi +b_{s}\right\vert ^{\nu }+t^{\max \left\{ 0,\nu -1\right\}
}\sum_{s=1}^{t}E\left\vert e_{s}\right\vert ^{\nu
}\prod\limits_{z=s}^{t-1}E\left\vert \varphi +b_{z+1}\right\vert ^{\nu }\leq
C_{0}\left( 1+t^{1+\max \left\{ 0,\nu -1\right\} }\right) ,
\end{equation*}%
whence again the desired result follows immediately. 
\end{proof}

The next two lemmas contain some anti-concentration bounds for the case of
nonstationary $X_{t}$.

\begin{lemma}
\label{7.4-ht}Under Assumptions \ref{as-1}, \ref{as-3} and \ref{as-4}\textit{%
(i)}, it holds that%
\begin{equation*}
P\left( \left\vert X_{t}\right\vert \leq t^{\alpha }\right) \leq 
\begin{array}{c}
C_{0}t^{\alpha }\left\{ \exp \left( -t^{3/2\left( 1+\nu \right) }\right)
+t^{-\left( \nu -2\right) /2\left( 1+\nu \right) }\right\} \text{ \ \ \ \ \
\ } \\ 
C_{0}t^{\alpha }\exp \left( -t\ln \left\vert \varphi \right\vert \right) 
\text{ \ \ \ \ \ \ \ \ \ \ \ \ \ \ \ \ \ \ \ \ \ \ \ \ \ \ \ \ \ \ \ \ \ }
\\ 
C_{0}t^{\alpha }\left\{ \exp \left( -t^{3/2\left( 1+\nu \right) }\right)
+t^{-2}+t^{-\left( \nu -2\right) /2\left( 1+\nu \right) }\right\} \\ 
C_{0}t^{1-\alpha \nu }+C_{1}t^{\alpha -1/2}\text{ \ \ \ \ \ \ \ \ \ \ \ \ \
\ \ \ \ \ \ \ \ \ \ \ \ \ \ \ \ \ \ \ }%
\end{array}%
\text{ for }%
\begin{array}{c}
E\ln \left\vert \varphi +b_{0}\right\vert >0\text{ with }P\left(
b_{0}=0\right) <1 \\ 
E\ln \left\vert \varphi +b_{0}\right\vert >0\text{ with }P\left(
b_{0}=0\right) =1 \\ 
E\ln \left\vert \varphi +b_{0}\right\vert =0\text{ with }P\left(
b_{0}=0\right) <1 \\ 
E\ln \left\vert \varphi +b_{0}\right\vert =0\text{ with }P\left(
b_{0}=0\right) =1%
\end{array}%
.
\end{equation*}
\end{lemma}

\begin{proof}
See Lemma 7.4 in \citet{HT2016}.
\end{proof}

Lemma \ref{7.4-ht} covers the cases $E\ln \left\vert \varphi
+b_{0}\right\vert >0$, $E\ln \left\vert \varphi +b_{0}\right\vert =0$ with
genuine random coefficient and $E\ln \left\vert \varphi +b_{0}\right\vert =0$
with no randomness and $E\left\vert e_{0}\right\vert ^{\nu ^{\prime
}}<\infty $, with $\nu ^{\prime }>2$. The next lemma is useful to study the
case of a non-random unit root process with infinite variance.

\begin{lemma}
\label{berkes}Under Assumptions \ref{as-1} and \ref{as-4}\textit{(ii)}, it
holds that%
\begin{equation*}
P\left( \left\vert X_{t}\right\vert \leq t^{\alpha }\right) \leq
C_{0}t^{1-\alpha \gamma -\gamma \epsilon }+C_{1}t^{\alpha -1/\gamma },
\end{equation*}%
for some $\epsilon >0$.
\end{lemma}

\begin{proof}
By Theorem 1 in \citet{berkes1986}, on a suitably larger space we can
construct two independent sequences of \textit{i.i.d.} random variables, say 
$\left\{ y_{i},i\geq 1\right\} $ and $\left\{ z_{i},i\geq 1\right\} $ such
that: $y_{i}$ and $z_{i}$ are both symmetric, the common characteristic
function of the $y_{i}$ is $\exp \left( -C_{0}\left\vert x\right\vert
^{\gamma }\right) $ with $C_{0}>0$, the $z_{i}$'s have common symmetric
distribution function $F_{z}\left( x\right) $ satisfying $1-F_{z}\left(
x\right) =\varsigma \left( x\right) x^{-\gamma }$ for $x\geq x_{0}$ and%
\begin{equation}
\sum_{i=1}^{t}\left( e_{i}-y_{i}-z_{i}\right) =O\left( t^{1/\gamma -\epsilon
^{\prime }}\right) ,  \label{berkes-sip}
\end{equation}%
for some $\epsilon ^{\prime }>0$. Note also that, using equation (2.32) in %
\citet{berkes1989} and Markov inequality, (\ref{berkes-sip}) entails the
following estimate%
\begin{equation}
P\left( \left\vert \sum_{i=1}^{t}\left( e_{i}-w_{i}\right) \right\vert \geq
t^{\alpha }\right) \leq C_{0}t^{1-\alpha \gamma -\gamma \epsilon }.
\label{berkes-dehling}
\end{equation}%
Let $w_{i}=y_{i}+z_{i}$; we can write%
\begin{equation*}
P\left( \left\vert X_{t}\right\vert \leq t^{\alpha }\right) \leq P\left(
\left\vert \sum_{i=1}^{t}w_{i}\right\vert \leq 2t^{\alpha }\right)
+C_{0}t^{1-\alpha \gamma -\gamma \epsilon }.
\end{equation*}%
Let now $Q\left( X;\lambda \right) =\sup_{x}P\left( x\leq X\leq x+\lambda
\right) $ denote the concentration function of a random variable $X$ (see
\citealp{petrov}), and let $Y=\sum_{i=1}^{t}y_{i}$ and $Z=\sum_{i=1}^{t}z_{i}$. Clearly%
\begin{equation*}
P\left( \left\vert \sum_{i=1}^{t}w_{i}\right\vert \leq t^{\alpha }\right)
\leq 2Q\left( Y+Z;t^{\alpha }\right) ,
\end{equation*}%
and by the independence between the $y_{i}$'s and the $z_{i}$'s,%
\begin{equation*}
Q\left( Y+Z;t^{\alpha }\right) \leq \min \left\{ Q\left( Y;t^{\alpha
}\right) ,Q\left( Z;t^{\alpha }\right) \right\} \leq Q\left( Y;t^{\alpha
}\right) .
\end{equation*}%
Now, given that the $y_{i}$'s have a distribution belonging in the stable
distribution family, we have%
\begin{equation}
P\left( x\leq \left\vert \sum_{i=1}^{t}y_{i}\right\vert \leq x+t^{\alpha
}\right) =P\left( xt^{-1/\gamma }\leq \left\vert y_{1}\right\vert \leq
\left( x+t^{\alpha }\right) t^{-1/\gamma }\right) ;  \label{stable}
\end{equation}%
further, since the characteristic function of the $y_{i}$s is integrable,
the density of the $y_{i}$s is bounded with upper bound $m_{y}$. Hence%
\begin{equation*}
\sup_{x}P\left( xt^{-1/\gamma }\leq \left\vert y_{1}\right\vert \leq \left(
x+t^{\alpha }\right) t^{-1/\gamma }\right) \leq \frac{1}{2}m_{y}t^{\alpha
-1/\gamma },
\end{equation*}%
which concludes the proof. 
\end{proof}

\begin{lemma}
\label{vp-bound}Let $a_{T}$ be a positive, real-valued sequence diverging to
positive infinity as $T\rightarrow \infty $. Under Assumptions \ref{as-1}, %
\ref{as-3} and \ref{as-4}, it holds that%
\begin{equation*}
P\left( v_{p}\geq a_{T}\right) \leq p^{-\min \left\{ 1,\nu /2\right\}
}a_{T}^{-\nu /2}\sum_{i=1}^{p}E\left\vert X_{i}\right\vert ^{\nu },
\end{equation*}%
for some $0<C_{0}<\infty $; $\nu $ is defined in Assumption \ref{as-1}, and,
when Assumption \ref{as-4}\textit{(ii) holds, it is} chosen so that $\nu
<\gamma $.
\end{lemma}

\begin{proof}
The lemma follows immediately upon noting that $P\left( v_{p}\geq
a_{T}\right) =P\left( v_{p}^{\nu /2}\geq a_{T}^{\nu /2}\right) $, and
applying convexity (when $\frac{\nu }{2}>1$) or the $C_{r}$-inequality
(otherwise) to $v_{p}^{\nu /2}$.
\end{proof}

\begin{lemma}
\label{expect-xt-ha}Under Assumptions \ref{as-1} and \ref{as-3}-\ref{p(T)},
it holds that, in all cases considered%
\begin{equation*}
\frac{1}{T}\sum_{t=1}^{T}E\left( \frac{v_{p}}{v_{p}+X_{t}^{2}}\right)
=O\left( T^{-\epsilon }\right) ,
\end{equation*}%
for some $\epsilon >0$.
\end{lemma}

\begin{proof}
Let $k>0$, and note that 
\begin{eqnarray*}
E\left( \frac{v_{p}}{v_{p}+X_{t}^{2}}\right) &=&E\left( \left. \frac{v_{p}}{%
v_{p}+X_{t}^{2}}\right\vert v_{p}\geq T^{k}\right) P\left( v_{p}\geq
T^{k}\right) +E\left( \left. \frac{v_{p}}{v_{p}+X_{t}^{2}}\right\vert
v_{p}<T^{k}\right) P\left( v_{p}<T^{k}\right) \\
&\leq &P\left( v_{p}\geq T^{k}\right) +E\left( \left. \frac{v_{p}}{%
v_{p}+X_{t}^{2}}\right\vert v_{p}<T^{k}\right) .
\end{eqnarray*}%
On account of Lemmas \ref{expect-vp-ha} and \ref{vp-bound}, it is immediate
to see that, in the worst case, $P\left( v_{p}\geq T^{k}\right) \leq
C_{0}T^{-k\nu /2}\left( \ln \ln T\right) ^{2}\left( \ln T\right) ^{\delta }$
for some $\delta >0$. Also%
\begin{eqnarray*}
&&E\left( \left. \frac{v_{p}}{v_{p}+X_{t}^{2}}\right\vert v_{p}<T^{k}\right)
\\
&\leq &E\left( \left. \frac{v_{p}}{v_{p}+X_{t}^{2}}\right\vert
v_{p}<T^{k},\left\vert X_{t}\right\vert >t^{\alpha }\right) P\left(
\left\vert X_{t}\right\vert >t^{\alpha }\right) +E\left( \left. \frac{v_{p}}{%
v_{p}+X_{t}^{2}}\right\vert v_{p}<T^{k},\left\vert X_{t}\right\vert \leq
t^{\alpha }\right) P\left( \left\vert X_{t}\right\vert \leq t^{\alpha
}\right) \\
&\leq &T^{k}t^{-2\alpha }+P\left( \left\vert X_{t}\right\vert \leq t^{\alpha
}\right) ;
\end{eqnarray*}%
hence%
\begin{equation*}
\frac{1}{T}\sum_{t=1}^{T}E\left( \frac{v_{p}}{v_{p}+X_{t}^{2}}\right) \leq
C_{0}T^{-k\nu /2}\left( \ln \ln T\right) ^{2}\left( \ln T\right) ^{\delta
}+C_{1}T^{k-2\alpha }+C_{2}\frac{1}{T}\sum_{t=1}^{T}P\left( \left\vert
X_{t}\right\vert \leq t^{\alpha }\right) .
\end{equation*}%
Using Lemmas \ref{7.4-ht} and \ref{berkes}, the desired result follows.
\end{proof}

Finally, we need the following lemma.

\begin{lemma}
\label{horvath-trapani}Consider a sequence $U_{T}$ for which $E\left\vert
U_{T}\right\vert \leq a_{T}$, where $a_{T}$ is a positive, monotonically
non-decreasing sequence. Then there exists a $C_{0}<\infty $ such that%
\begin{equation*}
\lim \sup_{T\rightarrow \infty }\frac{\left\vert U_{T}\right\vert }{%
a_{T}\left( \ln T\right) ^{2+\epsilon }}\leq C_{0}\text{ a.s.}
\end{equation*}
\end{lemma}

\begin{proof} 
By equation (2.3) in \citet{serfling1970}, it holds that $E\max_{1\leq t\leq
T}\left\vert U_{t}\right\vert \leq C_{1}a_{T}\ln T$. Therefore%
\begin{equation*}
\sum_{T=1}^{\infty }\frac{1}{T}P\left( \max_{1\leq t\leq T}\left\vert
U_{t}\right\vert \geq a_{T}\left( \ln T\right) ^{2+\epsilon }\right) \leq
C_{1}\sum_{T=1}^{\infty }\frac{1}{T}\frac{a_{T}\ln T}{a_{T}\left( \ln
T\right) ^{2+\epsilon }}<\infty .
\end{equation*}%
The desired result follows from the proof of Corollary 2.4 in \citet{cai2006}. 
\end{proof}

\section{Proofs\label{proofs}}

\begin{proof}[Proof of Theorem \protect\ref{dt-tilde-convergence}]
We begin with equation (\ref{th2.1-a}), which holds for $E\ln \left\vert
\varphi +b_{0}\right\vert <0$. The constant $C_{0}$ in the statement of the
theorem differs according as $E\overline{X}_{0}^{2}=\infty $\ or $E\overline{%
X}_{0}^{2}<\infty $. In the latter case, an immediate consequence of Lemma %
\ref{ergodic} is%
\begin{equation*}
\frac{1}{T-p}\sum_{t=p+1}^{T}\frac{v_{p}}{v_{p}+X_{t}^{2}}=E\left( \frac{E%
\overline{X}_{0}^{2}}{E\overline{X}_{0}^{2}+\overline{X}_{0}^{2}}\right)
+o\left( 1\right) ;
\end{equation*}%
on account of Assumption \ref{as-2}\textit{(i)}, we also have%
\begin{equation*}
E\left( \frac{E\overline{X}_{0}^{2}}{E\overline{X}_{0}^{2}+\overline{X}%
_{0}^{2}}\right) >0.
\end{equation*}%
When $E\overline{X}_{0}^{2}=\infty $, Lemma \ref{limit-one} yields 
\begin{equation*}
\frac{1}{T-p}\sum_{t=p+1}^{T}\frac{v_{p}}{v_{p}+X_{t}^{2}}=1+o\left(
1\right) ,
\end{equation*}%
whence the desired result. Finally, equation (\ref{th2.1-b}) is an immediate
consequence of Lemmas \ref{expect-xt-ha} and \ref{horvath-trapani}.
\end{proof}

\begin{proof}[Proof of Theorem \protect\ref{theta}]
The proof, and in particular condition (\ref{restriction}), are a refinement
of similar results in related papers, such as \citet{bandi2014} and %
\citet{HT16}. We begin by considering equation (\ref{asy-null}); in this
case, (\ref{th2.1-a}) entails $l_{T}=g\left( C_{0}\psi \left( T\right)
\right) $ a.s., so that, by (\ref{lt-null-as}), we can assume from now on
that $\lim_{T\rightarrow \infty }l_{T}=\infty $. Let $E^{\ast }$ denote the
expected value with respect to $P^{\ast }$, and $m_{G}$ the upper bound for
the density of $G$. It holds that: 
\begin{align}
R^{-1/2}\sum_{i=1}^{R}\left[ \zeta _{i}(u)-G\left( 0\right) \right] &
=R^{-1/2}\sum_{i=1}^{R}\left[ I\{\xi _{i}\leq 0\}-G\left( 0\right) \right]
+R^{-1/2}\sum_{i=1}^{R}\left[ G(ul_{T}^{-1})-G(0)\right]  \label{theta-r,t}
\\
& +R^{-1/2}\sum_{i=1}^{R}\left[ I\{0\leq \xi _{i}\leq
ul_{T}^{-1}\}-(G(ul_{T}^{-1})-G(0))\right]  \notag \\
& =I+II+III.  \notag
\end{align}%
Consider $III$, and note that the random variable $I\{0\leq \xi _{i}\leq
ul_{T}^{-1}\}$ has expected value $\Delta G(ul_{T}^{-1})=G(ul_{T}^{-1})-G(0)$
and variance $\Delta G(ul_{T}^{-1})\left[ 1-\Delta G(ul_{T}^{-1})\right] $.
Since the $\xi _{i}$s are \textit{i.i.d.}, it holds that 
\begin{eqnarray*}
&&E^{\ast }\int_{-\infty }^{\infty }\left\vert R^{-1/2}\sum_{i=1}^{R}\left[
I\{0\leq \xi _{i}\leq ul_{T}^{-1}\}-\Delta G(ul_{T}^{-1})\right] \right\vert
^{2}dF\left( u\right) \\
&=&\int_{-\infty }^{\infty }E^{\ast }\left[ I\{0\leq \xi _{1}\leq
ul_{T}^{-1}\}-\Delta G(ul_{T}^{-1})\right] ^{2}dF\left( u\right) \\
&=&\int_{-\infty }^{\infty }\Delta G(ul_{T}^{-1})\left[ 1-\Delta
G(ul_{T}^{-1})\right] dF\left( u\right) \\
&\leq &\int_{-\infty }^{\infty }\Delta G(ul_{T}^{-1})dF\left( u\right) \leq 
\frac{m_{G}}{l_{T}}\int_{-\infty }^{\infty }|u|dF(u);
\end{eqnarray*}%
thus, Assumption \ref{as-5} and Markov inequality entail that $%
III=o_{P^{\ast }}\left( 1\right) $. As far as $II$ is concerned, it
immediately follows that%
\begin{equation*}
\int_{-\infty }^{\infty }\left( R^{1/2}\left[ G(ul_{T}^{-1})-G(0)\right]
\right) ^{2}dF\left( u\right) \leq \frac{R}{l_{T}^{2}}m_{G}\int_{-\infty
}^{\infty }|u|^{2}dF(u);
\end{equation*}%
by Assumption \ref{as-5} and (\ref{restriction}), this also tends to zero.
Hence%
\begin{align*}
G\left( 0\right) \left( 1-G\left( 0\right) \right) \Theta _{T,R}&
=\int_{-\infty }^{\infty }\left\vert \frac{1}{R^{1/2}}\sum_{i=1}^{R}\left[
I\{\xi _{i}\leq 0\}-G\left( 0\right) \right] \right\vert
^{2}dF(u)+o_{P^{\ast }}(1) \\
& =\left\vert \frac{1}{R^{1/2}}\sum_{i=1}^{R}\left[ I\{\xi _{i}\leq
0\}-G\left( 0\right) \right] \right\vert ^{2}+o_{P^{\ast }}(1),
\end{align*}%
and therefore the result follows from the Central Limit Theorem for
Bernoulli random variables (see \citealp{chow2012}).

We now turn to (\ref{asy-alternative}). In this case, note that 
\begin{equation*}
E^{\ast }\left[ R^{-1/2}\sum_{i=1}^{R}(I\{\xi _{i}\leq ul_{T}^{-1}\}-G(0))%
\right] ^{2}=E^{\ast }\left[ I\{\xi _{1}\leq ul_{T}^{-1}\}-G(ul_{T}^{-1})%
\right] ^{2}+R\left\vert G(ul_{T}^{-1})-G(0))\right\vert ^{2}.
\end{equation*}%
Since $E^{\ast }\left[ I\{\xi _{1}\leq ul_{T}^{-1}\}-G(ul_{T}^{-1})\right]
^{2}<\infty $, and since (\ref{lt-alternative-as}) entails that we can
assume that $\lim_{T\rightarrow \infty }l_{T}=0$, by Markov inequality it
holds that%
\begin{equation*}
\int_{-\infty }^{\infty }\left[ R^{-1/2}\sum_{i=1}^{R}(I\{\xi _{i}\leq
ul_{T}^{-1}\}-G(0))\right] ^{2}dF\left( u\right) =R\left( 1-G\left( 0\right)
\right) +o\left( R\right) +O_{P^{\ast }}\left( 1\right) ,
\end{equation*}%
for almost all realisations of $\left\{ b_{t},e_{t},-\infty <t<\infty
\right\} $. Hence, (\ref{asy-alternative}) follows immediately. 
\end{proof}

\begin{proof}[Proof of Corollary \protect\ref{strong-rule}]
Let $Y_{s}=I\left( \Theta _{R,T}^{\left( s\right) }\leq c_{\alpha }\right) $%
, and note that this is an i.i.d. sequence with all moments finite. Equation
(\ref{strong-alt}) follows upon noting that%
\begin{equation*}
\frac{1}{S}\sum_{s=1}^{S}Y_{s}=\frac{1}{S}\sum_{s=1}^{S}\left(
Y_{s}-E^{^{\ast }}Y_{s}\right) +E^{^{\ast }}Y_{s};
\end{equation*}%
the first term converges to zero on account of the strong LLN (conditional
on the sample), whereas the second one, $E^{^{\ast }}Y_{s}$, drifts to zero
in the ordinary limit sense by Theorem \ref{theta}, again conditional on the
sample. As far as (\ref{strong-nul}) is concerned, by the Law of the
Iterated Logarithm, there exists a large, random $S_{0}$ such that, for all $%
S\geq S_{0}$%
\begin{equation*}
\sqrt{\frac{S}{\ln \ln S}}\frac{1}{S}\sum_{s=1}^{S}Y_{s}\geq -\sqrt{2\alpha
\left( 1-\alpha \right) }+\sqrt{\frac{S}{\ln \ln S}}E^{^{\ast }}Y_{s};
\end{equation*}%
also, by Theorem \ref{theta}, under $H_{0}$ it holds that $E^{^{\ast
}}Y_{s}=1-\alpha +o\left( 1\right) $, conditional on the sample. Equation (%
\ref{strong-nul}) follows immediately. 
\end{proof}

\begin{proof}[Proof of Theorem \protect\ref{dt-deterministics}]
The proof repeats the arguments above, and therefore we only report its main
passages. We let $d_{t}^{\ast }$ be equal to $d_{t}$ or $\widehat{d}%
_{t}-d_{t}$, depending on which case is considered. We begin with the case
of $H_{0}$, where we aim to show that%
\begin{equation}
\frac{1}{T-p}\sum_{t=p+1}^{T}\frac{Y_{t}^{2}}{v_{p}+Y_{t}^{2}}<1\text{ a.s.}
\label{null-lessthan1}
\end{equation}%
Consider first the case where $E\left\vert \overline{X}_{0}\right\vert
^{2}<\infty $. In this case, it is immediate to see that $v_{p}$ $=$ $%
p^{-1}\sum_{t=1}^{p}E\left( \overline{X}_{0}+d_{t}^{\ast }\right) ^{2}$ $+$ $%
o\left( 1\right) $, by the ergodic theorem. When $d_{t}^{\ast }=d_{t}$,
Assumption \ref{as-8}\textit{(ii)} ensures that $v_{p}>0$; otherwise,
Assumptions \ref{as-2}\textit{(i)} and \ref{as-8}\textit{(i)}(b) imply the
same result - note that it does not matter what limit $v_{p}$ converges to,
as long as it holds that such limit is positive. Then the same arguments as
in the proof of Lemma \ref{ergodic} yield (\ref{null-lessthan1}). Let $%
\overline{Y}_{t}=\overline{X}_{t}+d_{t}^{\ast }$. When $E\left\vert 
\overline{X}_{0}\right\vert ^{2}=\infty $, we have%
\begin{equation*}
\left\vert v_{p}-\frac{1}{p}\sum_{t=1}^{p}\overline{Y}_{t}^{2}\right\vert
\leq C_{0}\frac{1}{p}\sum_{t=1}^{p}\left\vert \overline{X}_{t}+d_{t}^{\ast
}\right\vert \left\vert \overline{X}_{t}-X_{t}\right\vert .
\end{equation*}%
We already know that $\sum_{t=1}^{p}\left\vert \overline{X}_{t}\right\vert
\left\vert \overline{X}_{t}-X_{t}\right\vert =O\left( 1\right) $. Also,
since $\sum_{t=1}^{p}\left\vert \overline{X}_{t}-X_{t}\right\vert
^{2}=O\left( 1\right) $, by the Cauchy-Schwartz inequality we have $%
\sum_{t=1}^{p}\left\vert d_{t}^{\ast }\right\vert \left\vert \overline{X}%
_{t}-X_{t}\right\vert \leq C_{o}\left( \sum_{t=1}^{p}\left\vert d_{t}^{\ast
}\right\vert ^{2}\right) ^{1/2}$; under Assumption \ref{as-8}\textit{(i)}%
(a), this entails that $\sum_{t=1}^{p}\left\vert d_{t}^{\ast }\right\vert
\left\vert \overline{X}_{t}-X_{t}\right\vert =O\left( p^{1/2}\right) $.
Conversely, when Assumption \ref{as-8}\textit{(i)}(b) holds, by using Lemma %
\ref{horvath-trapani} we have $\sum_{t=1}^{p}\left\vert d_{t}^{\ast
}\right\vert ^{2}=o\left( 1\right) $. Thus, in both cases%
\begin{equation*}
v_{p}=\frac{1}{p}\sum_{t=1}^{p}\overline{X}_{t}^{2}+\frac{1}{p}%
\sum_{t=1}^{p}d_{t}^{\ast 2}+\frac{2}{p}\sum_{t=1}^{p}d_{t}^{\ast }\overline{%
X}_{t}+o\left( 1\right) ;
\end{equation*}%
the second term is at most $O\left( 1\right) $, and therefore the third one
can be shown to be dominated by applying the Cauchy-Schwartz inequality.
Lemma \ref{vp_nonergodic} now entails that $v_{p}\rightarrow \infty $ a.s.;
also, by the same arguments as in (\ref{ergodic}) 
\begin{equation*}
\frac{1}{T-p}\sum_{t=p+1}^{T}\frac{Y_{t}^{2}}{v_{p}+Y_{t}^{2}}=\frac{1}{T-p}%
\sum_{t=p+1}^{T}\frac{\overline{Y}_{t}^{2}}{v_{p}+\overline{Y}_{t}^{2}}%
+o\left( 1\right) .
\end{equation*}%
Thus, for every $C_{0}$ there is a random $p_{0}$ such that, when $p\geq p_{0}$%
\begin{equation*}
\frac{1}{T-p}\sum_{t=p+1}^{T}\frac{\overline{Y}_{t}^{2}}{v_{p}+\overline{Y}%
_{t}^{2}}\leq \frac{1}{T-p}\sum_{t=p+1}^{T}\frac{\overline{Y}_{t}^{2}}{C_{0}+%
\overline{Y}_{t}^{2}}\text{ a.s.}
\end{equation*}%
Now, for $0<\delta <2$ such that $E\left\vert \overline{X}_{0}\right\vert
^{\delta }<\infty $\ we have%
\begin{eqnarray}
\frac{1}{T-p}\sum_{t=p+1}^{T}\frac{\overline{Y}_{t}^{2}}{C_{0}+\overline{Y}%
_{t}^{2}} &\leq &\frac{1}{T-p}\sum_{t=p+1}^{T}\left\vert \frac{\overline{Y}%
_{t}^{2}}{C_{0}+\overline{Y}_{t}^{2}}\right\vert ^{\delta /2}\leq
C_{0}^{-\delta /2}\frac{1}{T-p}\sum_{t=p+1}^{T}\left\vert \overline{Y}%
_{t}\right\vert ^{\delta }  \notag \\
&\leq &C_{1}C_{0}^{-\delta /2}\frac{1}{T-p}\sum_{t=p+1}^{T}\left( \left\vert 
\overline{X}_{t}\right\vert ^{\delta }+\left\vert d_{t}^{\ast }\right\vert
^{\delta }\right)   \notag \\
&=&C_{1}C_{0}^{-\delta /2}E\left\vert \overline{X}_{0}\right\vert ^{\delta
}+C_{1}C_{0}^{-\delta /2}\frac{1}{T-p}\sum_{t=p+1}^{T}\left\vert d_{t}^{\ast
}\right\vert ^{\delta }\leq C_{2}C_{0}^{-\delta /2},  \label{limit-1}
\end{eqnarray}%
where $C_{2}<\infty $. This follows directly when Assumption \ref{as-8}%
\textit{(i)}(a) holds; under Assumption \ref{as-8}\textit{(i)}(b), it can be
shown by elementary arguments that $E\left\vert \sum_{t=p+1}^{T}\left\vert
d_{t}^{\ast }\right\vert ^{\delta }\right\vert $ $\leq $ $%
\sum_{t=p+1}^{T}E\left\vert d_{t}^{\ast }\right\vert ^{2}$ $\leq $ $%
C_{0}T^{1-\epsilon _{1}}$, so that Lemma \ref{horvath-trapani} ensures that $%
\left( T-p\right) ^{-1}\sum_{t=p+1}^{T}\left\vert d_{t}^{\ast }\right\vert
^{\delta }$ $=$ $o\left( 1\right) $, whence (\ref{limit-1}). In conclusion,
for any $\epsilon >0$ there is a random $p_{0}$ such that for all $p\geq p_{0}$%
\begin{equation*}
\frac{1}{T-p}\sum_{t=p+1}^{T}\frac{Y_{t}^{2}}{v_{p}+Y_{t}^{2}}\leq \epsilon 
\text{ a.s.,}
\end{equation*}%
which proves (\ref{null-lessthan1}) even in the case of infinite variance.
Thus, the same passages as above yield that, under $H_{0}$, $D_{T}^{\ast
}=C_{0}+o\left( 1\right) $ for some $0<C_{0}<\infty $.

Under $H_{A}$, the proof that $D_{T}^{\ast }=O\left( T^{-\epsilon }\right) $
follows immediately if we show that Lemmas \ref{7.4-ht} and \ref{berkes}
hold. However, this can be easily verified by noting that $P\left(
\left\vert Y_{t}\right\vert \leq t^{\alpha }\right) \leq P\left( \left\vert
X_{t}\right\vert \leq t^{\alpha }+\left\vert d_{t}^{\ast }\right\vert
\right) $.
\end{proof}

{\small {\ }}

{\small {\setlength{\bibsep}{.2cm} 
\bibliographystyle{chicago}
\bibliography{LTbiblio}
}}

\begin{landscape}

\begin{center}
\begin{tabular}{ll|llll|llll|llll}
&  & \multicolumn{4}{|c|}{$e_{t}\sim N\left( 0,1\right) $} & 
\multicolumn{4}{|c|}{$e_{t}\sim t_{2}$} & \multicolumn{4}{|c}{$e_{t}\sim
t_{1}$} \\ 
&  & $250\ $ & $500\ $ & $1000$ & $2000$ & $250\ $ & $500\ $ & $1000$ & $2000
$ & $250\ $ & $500\ $ & $1000$ & $2000$ \\ 
$\varphi $ & $\sigma _{b}^{2}$ &  &  &  &  &  &  &  &  &  &  &  &  \\ 
$1.05$ & \multicolumn{1}{c|}{%
\begin{tabular}{l}
$\underset{}{0}$ \\ 
$\underset{}{0.10}$ \\ 
$\underset{}{0.25}$%
\end{tabular}%
} & $%
\begin{array}{c}
\underset{\left( 0.00\right) }{0.989} \\ 
\underset{\left( 0.12\right) }{0.389} \\ 
\underset{\left( 0.73\right) }{0.071}%
\end{array}%
$ & $%
\begin{array}{c}
\underset{\left( 0.00\right) }{0.981} \\ 
\underset{\left( 0.09\right) }{0.410} \\ 
\underset{\left( 0.93\right) }{0.060}%
\end{array}%
$ & $%
\begin{array}{c}
\underset{\left( 0.00\right) }{0.998} \\ 
\underset{\left( 0.07\right) }{0.519} \\ 
\underset{\left( 0.94\right) }{0.059}%
\end{array}%
$ & $%
\begin{array}{c}
\underset{\left( 0.00\right) }{1.000} \\ 
\underset{\left( 0.05\right) }{0.628} \\ 
\underset{\left( 0.96\right) }{0.051}%
\end{array}%
$ & $%
\begin{array}{c}
\underset{\left( 0.00\right) }{0.889} \\ 
\underset{\left( 0.11\right) }{0.473} \\ 
\underset{\left( 0.64\right) }{0.065}%
\end{array}%
$ & $%
\begin{array}{c}
\underset{\left( 0.00\right) }{0.990} \\ 
\underset{\left( 0.08\right) }{0.413} \\ 
\underset{\left( 0.81\right) }{0.055}%
\end{array}%
$ & $%
\begin{array}{c}
\underset{\left( 0.00\right) }{0.998} \\ 
\underset{\left( 0.07\right) }{0.556} \\ 
\underset{\left( 0.90\right) }{0.057}%
\end{array}%
$ & $%
\begin{array}{c}
\underset{\left( 0.00\right) }{1.000} \\ 
\underset{\left( 0.05\right) }{0.619} \\ 
\underset{\left( 0.99\right) }{0.052}%
\end{array}%
$ & $%
\begin{array}{c}
\underset{\left( 0.00\right) }{0.647} \\ 
\underset{\left( 0.12\right) }{0.266} \\ 
\underset{\left( 0.50\right) }{0.055}%
\end{array}%
$ & $%
\begin{array}{c}
\underset{\left( 0.00\right) }{0.995} \\ 
\underset{\left( 0.09\right) }{0.398} \\ 
\underset{\left( 0.78\right) }{0.050}%
\end{array}%
$ & $%
\begin{array}{c}
\underset{\left( 0.00\right) }{1.000} \\ 
\underset{\left( 0.07\right) }{0.434} \\ 
\underset{\left( 0.91\right) }{0.060}%
\end{array}%
$ & $%
\begin{array}{c}
\underset{\left( 0.00\right) }{1.000} \\ 
\underset{\left( 0.04\right) }{0.579} \\ 
\underset{\left( 0.96\right) }{0.055}%
\end{array}%
$ \\ 
$1$ & 
\begin{tabular}{l}
$\underset{}{0}$ \\ 
$\underset{}{0.10}$ \\ 
$\underset{}{0.25}$%
\end{tabular}
& $%
\begin{array}{c}
\underset{\left( 0.06\right) }{0.484} \\ 
\underset{\left( 0.90\right) }{0.068} \\ 
\underset{\left( 0.93\right) }{0.050}%
\end{array}%
$ & $%
\begin{array}{c}
\underset{\left( 0.05\right) }{0.525} \\ 
\underset{\left( 0.93\right) }{0.059} \\ 
\underset{\left( 0.93\right) }{0.056}%
\end{array}%
$ & $%
\begin{array}{c}
\underset{\left( 0.03\right) }{0.687} \\ 
\underset{\left( 0.94\right) }{0.059} \\ 
\underset{\left( 0.97\right) }{0.059}%
\end{array}%
$ & $%
\begin{array}{c}
\underset{\left( 0.01\right) }{0.754} \\ 
\underset{\left( 0.97\right) }{0.052} \\ 
\underset{\left( 0.99\right) }{0.051}%
\end{array}%
$ & $%
\begin{array}{c}
\underset{\left( 0.06\right) }{0.350} \\ 
\underset{\left( 0.80\right) }{0.074} \\ 
\underset{\left( 0.78\right) }{0.060}%
\end{array}%
$ & $%
\begin{array}{c}
\underset{\left( 0.04\right) }{0.564} \\ 
\underset{\left( 0.82\right) }{0.061} \\ 
\underset{\left( 0.84\right) }{0.053}%
\end{array}%
$ & $%
\begin{array}{c}
\underset{\left( 0.01\right) }{0.711} \\ 
\underset{\left( 0.89\right) }{0.057} \\ 
\underset{\left( 0.88\right) }{0.057}%
\end{array}%
$ & $%
\begin{array}{c}
\underset{\left( 0.02\right) }{0.796} \\ 
\underset{\left( 0.90\right) }{0.053} \\ 
\underset{\left( 0.91\right) }{0.051}%
\end{array}%
$ & $%
\begin{array}{c}
\underset{\left( 0.04\right) }{0.432} \\ 
\underset{\left( 0.81\right) }{0.090} \\ 
\underset{\left( 0.78\right) }{0.052}%
\end{array}%
$ & $%
\begin{array}{c}
\underset{\left( 0.04\right) }{0.537} \\ 
\underset{\left( 0.77\right) }{0.056} \\ 
\underset{\left( 0.82\right) }{0.047}%
\end{array}%
$ & $%
\begin{array}{c}
\underset{\left( 0.01\right) }{0.781} \\ 
\underset{\left( 0.87\right) }{0.059} \\ 
\underset{\left( 0.90\right) }{0.058}%
\end{array}%
$ & $%
\begin{array}{c}
\underset{\left( 0.00\right) }{0.857} \\ 
\underset{\left( 0.92\right) }{0.051} \\ 
\underset{\left( 0.92\right) }{0.051}%
\end{array}%
$ \\ 
$0.95$ & 
\begin{tabular}{l}
$\underset{}{0}$ \\ 
$\underset{}{0.10}$ \\ 
$\underset{}{0.25}$%
\end{tabular}
& $%
\begin{array}{c}
\underset{\left( 0.79\right) }{0.054} \\ 
\underset{\left( 0.90\right) }{0.049} \\ 
\underset{\left( 0.91\right) }{0.048}%
\end{array}%
$ & $%
\begin{array}{c}
\underset{\left( 0.88\right) }{0.056} \\ 
\underset{\left( 0.90\right) }{0.056} \\ 
\underset{\left( 0.91\right) }{0.055}%
\end{array}%
$ & $%
\begin{array}{c}
\underset{\left( 0.90\right) }{0.060} \\ 
\underset{\left( 0.94\right) }{0.058} \\ 
\underset{\left( 0.94\right) }{0.058}%
\end{array}%
$ & $%
\begin{array}{c}
\underset{\left( 0.92\right) }{0.051} \\ 
\underset{\left( 0.96\right) }{0.051} \\ 
\underset{\left( 0.96\right) }{0.051}%
\end{array}%
$ & $%
\begin{array}{c}
\underset{\left( 0.60\right) }{0.051} \\ 
\underset{\left( 0.64\right) }{0.058} \\ 
\underset{\left( 0.71\right) }{0.052}%
\end{array}%
$ & $%
\begin{array}{c}
\underset{\left( 0.71\right) }{0.053} \\ 
\underset{\left( 0.73\right) }{0.053} \\ 
\underset{\left( 0.81\right) }{0.053}%
\end{array}%
$ & $%
\begin{array}{c}
\underset{\left( 0.79\right) }{0.057} \\ 
\underset{\left( 0.82\right) }{0.057} \\ 
\underset{\left( 0.86\right) }{0.057}%
\end{array}%
$ & $%
\begin{array}{c}
\underset{\left( 0.90\right) }{0.055} \\ 
\underset{\left( 0.90\right) }{0.051} \\ 
\underset{\left( 0.94\right) }{0.051}%
\end{array}%
$ & $%
\begin{array}{c}
\underset{\left( 0.52\right) }{0.053} \\ 
\underset{\left( 0.55\right) }{0.052} \\ 
\underset{\left( 0.58\right) }{0.051}%
\end{array}%
$ & $%
\begin{array}{c}
\underset{\left( 0.66\right) }{0.054} \\ 
\underset{\left( 0.57\right) }{0.047} \\ 
\underset{\left( 0.59\right) }{0.047}%
\end{array}%
$ & $%
\begin{array}{c}
\underset{\left( 0.70\right) }{0.056} \\ 
\underset{\left( 0.67\right) }{0.058} \\ 
\underset{\left( 0.76\right) }{0.058}%
\end{array}%
$ & $%
\begin{array}{c}
\underset{\left( 0.85\right) }{0.060} \\ 
\underset{\left( 0.95\right) }{0.051} \\ 
\underset{\left( 0.99\right) }{0.051}%
\end{array}%
$ \\ 
$0.5$ & 
\begin{tabular}{l}
$\underset{}{0}$ \\ 
$\underset{}{0.10}$ \\ 
$\underset{}{0.25}$%
\end{tabular}
& $%
\begin{array}{c}
\underset{\left( 0.90\right) }{0.048} \\ 
\underset{\left( 0.91\right) }{0.048} \\ 
\underset{\left( 0.93\right) }{0.048}%
\end{array}%
$ & $%
\begin{array}{c}
\underset{\left( 0.91\right) }{0.055} \\ 
\underset{\left( 0.96\right) }{0.055} \\ 
\underset{\left( 0.95\right) }{0.055}%
\end{array}%
$ & $%
\begin{array}{c}
\underset{\left( 0.95\right) }{0.058} \\ 
\underset{\left( 0.95\right) }{0.058} \\ 
\underset{\left( 0.98\right) }{0.058}%
\end{array}%
$ & $%
\begin{array}{c}
\underset{\left( 1.00\right) }{0.051} \\ 
\underset{\left( 0.99\right) }{0.051} \\ 
\underset{\left( 0.99\right) }{0.051}%
\end{array}%
$ & $%
\begin{array}{c}
\underset{\left( 0.92\right) }{0.050} \\ 
\underset{\left( 0.95\right) }{0.050} \\ 
\underset{\left( 0.93\right) }{0.050}%
\end{array}%
$ & $%
\begin{array}{c}
\underset{\left( 0.92\right) }{0.053} \\ 
\underset{\left( 0.91\right) }{0.053} \\ 
\underset{\left( 0.92\right) }{0.053}%
\end{array}%
$ & $%
\begin{array}{c}
\underset{\left( 0.96\right) }{0.057} \\ 
\underset{\left( 0.95\right) }{0.057} \\ 
\underset{\left( 0.96\right) }{0.057}%
\end{array}%
$ & $%
\begin{array}{c}
\underset{\left( 0.99\right) }{0.051} \\ 
\underset{\left( 0.99\right) }{0.051} \\ 
\underset{\left( 0.99\right) }{0.051}%
\end{array}%
$ & $%
\begin{array}{c}
\underset{\left( 0.92\right) }{0.051} \\ 
\underset{\left( 0.95\right) }{0.051} \\ 
\underset{\left( 0.95\right) }{0.051}%
\end{array}%
$ & $%
\begin{array}{c}
\underset{\left( 0.90\right) }{0.047} \\ 
\underset{\left( 0.91\right) }{0.047} \\ 
\underset{\left( 0.95\right) }{0.047}%
\end{array}%
$ & $%
\begin{array}{c}
\underset{\left( 0.93\right) }{0.057} \\ 
\underset{\left( 0.99\right) }{0.058} \\ 
\underset{\left( 0.99\right) }{0.058}%
\end{array}%
$ & $%
\begin{array}{c}
\underset{\left( 0.98\right) }{0.051} \\ 
\underset{\left( 0.99\right) }{0.051} \\ 
\underset{\left( 0.99\right) }{0.051}%
\end{array}%
$ \\ 
$0.0$ & 
\begin{tabular}{l}
$\underset{}{0}$ \\ 
$\underset{}{0.10}$ \\ 
$\underset{}{0.25}$%
\end{tabular}
& $%
\begin{array}{c}
\underset{\left( 0.92\right) }{0.048} \\ 
\underset{\left( 0.93\right) }{0.048} \\ 
\underset{\left( 0.93\right) }{0.048}%
\end{array}%
$ & $%
\begin{array}{c}
\underset{\left( 0.93\right) }{0.056} \\ 
\underset{\left( 0.95\right) }{0.055} \\ 
\underset{\left( 0.96\right) }{0.055}%
\end{array}%
$ & $%
\begin{array}{c}
\underset{\left( 0.96\right) }{0.058} \\ 
\underset{\left( 0.98\right) }{0.058} \\ 
\underset{\left( 0.98\right) }{0.058}%
\end{array}%
$ & $%
\begin{array}{c}
\underset{\left( 0.99\right) }{0.051} \\ 
\underset{\left( 0.99\right) }{0.051} \\ 
\underset{\left( 0.99\right) }{0.051}%
\end{array}%
$ & $%
\begin{array}{c}
\underset{\left( 0.95\right) }{0.050} \\ 
\underset{\left( 0.95\right) }{0.050} \\ 
\underset{\left( 0.95\right) }{0.050}%
\end{array}%
$ & $%
\begin{array}{c}
\underset{\left( 0.93\right) }{0.053} \\ 
\underset{\left( 0.94\right) }{0.053} \\ 
\underset{\left( 0.95\right) }{0.053}%
\end{array}%
$ & $%
\begin{array}{c}
\underset{\left( 0.96\right) }{0.057} \\ 
\underset{\left( 0.95\right) }{0.057} \\ 
\underset{\left( 0.96\right) }{0.057}%
\end{array}%
$ & $%
\begin{array}{c}
\underset{\left( 0.99\right) }{0.051} \\ 
\underset{\left( 0.99\right) }{0.051} \\ 
\underset{\left( 0.99\right) }{0.051}%
\end{array}%
$ & $%
\begin{array}{c}
\underset{\left( 0.94\right) }{0.051} \\ 
\underset{\left( 0.95\right) }{0.051} \\ 
\underset{\left( 0.95\right) }{0.051}%
\end{array}%
$ & $%
\begin{array}{c}
\underset{\left( 0.94\right) }{0.047} \\ 
\underset{\left( 0.92\right) }{0.047} \\ 
\underset{\left( 0.93\right) }{0.047}%
\end{array}%
$ & $%
\begin{array}{c}
\underset{\left( 0.95\right) }{0.057} \\ 
\underset{\left( 0.99\right) }{0.058} \\ 
\underset{\left( 0.99\right) }{0.058}%
\end{array}%
$ & $%
\begin{array}{c}
\underset{\left( 0.99\right) }{0.051} \\ 
\underset{\left( 0.99\right) }{0.051} \\ 
\underset{\left( 0.99\right) }{0.051}%
\end{array}%
$%
\end{tabular}

\bigskip
\end{center}

\textbf{Table 1. Empirical rejection frequencies for the test for }$%
H_{0}:X_{t}$\textbf{\ is stationary. }\newpage

\begin{center}
\begin{tabular}{cl|llll|llll|llll}
&  & \multicolumn{4}{|c|}{$e_{t}\sim N\left( 0,1\right) $} & 
\multicolumn{4}{|c|}{$e_{t}\sim t_{2}$} & \multicolumn{4}{|c}{$e_{t}\sim
t_{1}$} \\ 
& $T$ & $250$ \  & $500$ \  & $1000$ & $2000$ & $250\ $ & $500\ $ & $1000$ & 
$2000$ & $250\ $ & $500\ $ & $1000$ & $2000$ \\ 
$\left( \varphi ,\sigma _{b}^{2}\right) $ &  &  &  &  &  &  &  &  &  &  &  & 
&  \\ 
$\left( 0.2,3.6190\right) $ &  & $\underset{\left( 0.12\right) }{0.574}$ & $%
\underset{\left( 0.09\right) }{0.657}$ & $\underset{\left( 0.07\right) }{%
0.789}$ & $\underset{\left( 0.02\right) }{0.913}$ & $\underset{\left(
0.12\right) }{0.994}$ & $\underset{\left( 0.06\right) }{1.000}$ & $\underset{%
\left( 0.02\right) }{1.000}$ & $\underset{\left( 0.01\right) }{1.000}$ & $%
\underset{\left( 0.10\right) }{0.997}$ & $\underset{\left( 0.06\right) }{%
1.000}$ & $\underset{\left( 0.03\right) }{1.000}$ & \multicolumn{1}{l}{$%
\underset{\left( 0.01\right) }{1.000}$} \\ 
$\left( 0.3,3.5556\right) $ &  & $\underset{\left( 0.12\right) }{0.556}$ & $%
\underset{\left( 0.10\right) }{0.662}$ & $\underset{\left( 0.08\right) }{%
0.781}$ & $\underset{\left( 0.02\right) }{0.903}$ & $\underset{\left(
0.11\right) }{0.985}$ & $\underset{\left( 0.08\right) }{0.999}$ & $\underset{%
\left( 0.03\right) }{1.000}$ & $\underset{\left( 0.02\right) }{1.000}$ & $%
\underset{\left( 0.09\right) }{0.994}$ & $\underset{\left( 0.07\right) }{%
1.000}$ & $\underset{\left( 0.02\right) }{1.000}$ & \multicolumn{1}{l}{$%
\underset{\left( 0.01\right) }{1.000}$} \\ 
$\left( 0.4,3.4460\right) $ &  & $\underset{\left( 0.11\right) }{0.551}$ & $%
\underset{\left( 0.10\right) }{0.624}$ & $\underset{\left( 0.07\right) }{%
0.756}$ & $\underset{\left( 0.02\right) }{0.964}$ & $\underset{\left(
0.10\right) }{0.988}$ & $\underset{\left( 0.07\right) }{0.999}$ & $\underset{%
\left( 0.02\right) }{1.000}$ & $\underset{\left( 0.02\right) }{1.000}$ & $%
\underset{\left( 0.09\right) }{0.998}$ & $\underset{\left( 0.06\right) }{%
0.999}$ & $\underset{\left( 0.03\right) }{1.000}$ & \multicolumn{1}{l}{$%
\underset{\left( 0.00\right) }{1.000}$} \\ 
$\left( 0.5,3.3390\right) $ &  & $\underset{\left( 0.12\right) }{0.553}$ & $%
\underset{\left( 0.09\right) }{0.614}$ & $\underset{\left( 0.07\right) }{%
0.747}$ & $\underset{\left( 0.02\right) }{0.957}$ & $\underset{\left(
0.10\right) }{0.994}$ & $\underset{\left( 0.07\right) }{0.999}$ & $\underset{%
\left( 0.02\right) }{1.000}$ & $\underset{\left( 0.01\right) }{1.000}$ & $%
\underset{\left( 0.10\right) }{0.974}$ & $\underset{\left( 0.06\right) }{%
0.999}$ & $\underset{\left( 0.03\right) }{1.000}$ & \multicolumn{1}{l}{$%
\underset{\left( 0.01\right) }{1.000}$} \\ 
$\left( 0.6,3.2245\right) $ &  & $\underset{\left( 0.10\right) }{0.556}$ & $%
\underset{\left( 0.10\right) }{0.614}$ & $\underset{\left( 0.08\right) }{%
0.754}$ & $\underset{\left( 0.01\right) }{0.968}$ & $\underset{\left(
0.11\right) }{0.996}$ & $\underset{\left( 0.07\right) }{0.999}$ & $\underset{%
\left( 0.04\right) }{1.000}$ & $\underset{\left( 0.01\right) }{1.000}$ & $%
\underset{\left( 0.09\right) }{0.998}$ & $\underset{\left( 0.07\right) }{%
1.000}$ & $\underset{\left( 0.04\right) }{1.000}$ & \multicolumn{1}{l}{$%
\underset{\left( 0.00\right) }{1.000}$} \\ 
$\left( 0.7,3.1310\right) $ &  & $\underset{\left( 0.11\right) }{0.553}$ & $%
\underset{\left( 0.10\right) }{0.670}$ & $\underset{\left( 0.07\right) }{%
0.792}$ & $\underset{\left( 0.02\right) }{0.906}$ & $\underset{\left(
0.11\right) }{0.993}$ & $\underset{\left( 0.06\right) }{0.998}$ & $\underset{%
\left( 0.03\right) }{1.000}$ & $\underset{\left( 0.02\right) }{1.000}$ & $%
\underset{\left( 0.09\right) }{0.919}$ & $\underset{\left( 0.06\right) }{%
0.999}$ & $\underset{\left( 0.02\right) }{1.000}$ & \multicolumn{1}{l}{$%
\underset{\left( 0.01\right) }{1.000}$} \\ 
$\left( 0.8,2.8650\right) $ &  & $\underset{\left( 0.12\right) }{0.537}$ & $%
\underset{\left( 0.09\right) }{0.687}$ & $\underset{\left( 0.08\right) }{%
0.701}$ & $\underset{\left( 0.01\right) }{0.900}$ & $\underset{\left(
0.10\right) }{0.991}$ & $\underset{\left( 0.06\right) }{0.998}$ & $\underset{%
\left( 0.02\right) }{1.000}$ & $\underset{\left( 0.02\right) }{1.000}$ & $%
\underset{\left( 0.09\right) }{0.999}$ & $\underset{\left( 0.07\right) }{%
0.996}$ & $\underset{\left( 0.02\right) }{1.000}$ & \multicolumn{1}{l}{$%
\underset{\left( 0.01\right) }{1.000}$} \\ 
$\left( 0.9,2.6815\right) $ &  & $\underset{\left( 0.12\right) }{0.545}$ & $%
\underset{\left( 0.09\right) }{0.619}$ & $\underset{\left( 0,08\right) }{%
0.731}$ & $\underset{\left( 0.01\right) }{0.924}$ & $\underset{\left(
0.10\right) }{0.992}$ & $\underset{\left( 0.07\right) }{0.998}$ & $\underset{%
\left( 0.02\right) }{1.000}$ & $\underset{\left( 0.01\right) }{1.000}$ & $%
\underset{\left( 0.10\right) }{0.998}$ & $\underset{\left( 0.06\right) }{%
1.000}$ & $\underset{\left( 0.03\right) }{1.000}$ & \multicolumn{1}{l}{$%
\underset{\left( 0.01\right) }{1.000}$} \\ 
$\left( 1.0,2.4440\right) $ &  & $\underset{\left( 0.12\right) }{0.554}$ & $%
\underset{\left( 0.09\right) }{0.638}$ & $\underset{\left( 0.07\right) }{%
0.748}$ & $\underset{\left( 0.02\right) }{0.930}$ & $\underset{\left(
0.10\right) }{0.989}$ & $\underset{\left( 0.06\right) }{0.999}$ & $\underset{%
\left( 0.03\right) }{1.000}$ & $\underset{\left( 0.01\right) }{1.000}$ & $%
\underset{\left( 0.10\right) }{0.994}$ & $\underset{\left( 0.06\right) }{%
1.000}$ & $\underset{\left( 0.02\right) }{1.000}$ & \multicolumn{1}{l}{$%
\underset{\left( 0.00\right) }{1.000}$}%
\end{tabular}

\bigskip
\end{center}

\textbf{Table 2. Empirical rejection frequencies for the test for }$%
H_{0}:X_{t}$\textbf{\ is stationary. The cases considered in the table are
boundary cases in which }$E\ln \left\vert \varphi +b_{0}\right\vert =0$%
\textbf{\ - thus, }$X_{t}$\textbf{\ is nonstationary, and the figures in the
table should be interpreted as an assessment of the size of the test. }%
\newpage

\begin{center}
\begin{tabular}{ll|llll|llll|llll}
&  & \multicolumn{4}{|c|}{$e_{t}\sim N\left( 0,1\right) $} & 
\multicolumn{4}{|c|}{$e_{t}\sim t_{2}$} & \multicolumn{4}{|c}{$e_{t}\sim
t_{1}$} \\ 
&  & $250\ $ & $500\ $ & $1000$ & $2000$ & $250\ $ & $500\ $ & $1000$ & $%
2000 $ & $250\ $ & $500\ $ & $1000$ & $2000$ \\ 
$\varphi $ & $\sigma _{b}^{2}$ &  &  &  &  &  &  &  &  &  &  &  &  \\ 
$1.05$ & 
\begin{tabular}{l}
$0$ \\ 
$0.10$ \\ 
$0.25$%
\end{tabular}
& $%
\begin{array}{c}
0.052 \\ 
0.056 \\ 
0.096%
\end{array}%
$ & $%
\begin{array}{c}
0.052 \\ 
0.064 \\ 
0.219%
\end{array}%
$ & $%
\begin{array}{c}
0.049 \\ 
0.055 \\ 
0.290%
\end{array}%
$ & $%
\begin{array}{c}
0.051 \\ 
0.062 \\ 
0.465%
\end{array}%
$ & $%
\begin{array}{c}
0.053 \\ 
0.054 \\ 
0.102%
\end{array}%
$ & $%
\begin{array}{c}
0.054 \\ 
0.062 \\ 
0.180%
\end{array}%
$ & $%
\begin{array}{c}
0.052 \\ 
0.053 \\ 
0.240%
\end{array}%
$ & $%
\begin{array}{c}
0.058 \\ 
0.063 \\ 
0.483%
\end{array}%
$ & $%
\begin{array}{c}
0.053 \\ 
0.054 \\ 
0.064%
\end{array}%
$ & $%
\begin{array}{c}
0.055 \\ 
0.108 \\ 
0.129%
\end{array}%
$ & $%
\begin{array}{c}
0.053 \\ 
0.057 \\ 
0.345%
\end{array}%
$ & $%
\begin{array}{c}
0.060 \\ 
0.073 \\ 
0.489%
\end{array}%
$ \\ 
$1$ & 
\begin{tabular}{l}
$0$ \\ 
$0.10$ \\ 
$0.25$%
\end{tabular}
& $%
\begin{array}{c}
0.055 \\ 
0.096 \\ 
0.144%
\end{array}%
$ & $%
\begin{array}{c}
0.057 \\ 
0.306 \\ 
0.459%
\end{array}%
$ & $%
\begin{array}{c}
0.055 \\ 
0.469 \\ 
0.536%
\end{array}%
$ & $%
\begin{array}{c}
0.055 \\ 
0.639 \\ 
0.746%
\end{array}%
$ & $%
\begin{array}{c}
0.053 \\ 
0.091 \\ 
0.138%
\end{array}%
$ & $%
\begin{array}{c}
0.054 \\ 
0.209 \\ 
0.495%
\end{array}%
$ & $%
\begin{array}{c}
0.052 \\ 
0.345 \\ 
0.579%
\end{array}%
$ & $%
\begin{array}{c}
0.058 \\ 
0.622 \\ 
0.723%
\end{array}%
$ & $%
\begin{array}{c}
0.053 \\ 
0.169 \\ 
0.147%
\end{array}%
$ & $%
\begin{array}{c}
0.055 \\ 
0.260 \\ 
0.463%
\end{array}%
$ & $%
\begin{array}{c}
0.052 \\ 
0.442 \\ 
0.509%
\end{array}%
$ & $%
\begin{array}{c}
0.060 \\ 
0.630 \\ 
0.712%
\end{array}%
$ \\ 
$0.95$ & 
\begin{tabular}{l}
$0$ \\ 
$0.10$ \\ 
$0.25$%
\end{tabular}
& $%
\begin{array}{c}
0.222 \\ 
0.229 \\ 
0.238%
\end{array}%
$ & $%
\begin{array}{c}
0.497 \\ 
0.505 \\ 
0.535%
\end{array}%
$ & $%
\begin{array}{c}
0.579 \\ 
0.600 \\ 
0.630%
\end{array}%
$ & $%
\begin{array}{c}
0.780 \\ 
0.813 \\ 
0.833%
\end{array}%
$ & $%
\begin{array}{c}
0.159 \\ 
0.221 \\ 
0.233%
\end{array}%
$ & $%
\begin{array}{c}
0.421 \\ 
0.593 \\ 
0.541%
\end{array}%
$ & $%
\begin{array}{c}
0.549 \\ 
0.464 \\ 
0.674%
\end{array}%
$ & $%
\begin{array}{c}
0.886 \\ 
0.829 \\ 
0.831%
\end{array}%
$ & $%
\begin{array}{c}
0.154 \\ 
0.218 \\ 
0.256%
\end{array}%
$ & $%
\begin{array}{c}
0.467 \\ 
0.522 \\ 
0.505%
\end{array}%
$ & $%
\begin{array}{c}
0.582 \\ 
0.548 \\ 
0.630%
\end{array}%
$ & $%
\begin{array}{c}
0.754 \\ 
0.763 \\ 
0.802%
\end{array}%
$ \\ 
$0.5$ & 
\begin{tabular}{l}
$0$ \\ 
$0.10$ \\ 
$0.25$%
\end{tabular}
& $%
\begin{array}{c}
0.467 \\ 
0.457 \\ 
0.430%
\end{array}%
$ & $%
\begin{array}{c}
0.841 \\ 
0.833 \\ 
0.815%
\end{array}%
$ & $%
\begin{array}{c}
0.907 \\ 
0.900 \\ 
0.877%
\end{array}%
$ & $%
\begin{array}{c}
0.986 \\ 
0.982 \\ 
0.975%
\end{array}%
$ & $%
\begin{array}{c}
0.557 \\ 
0.533 \\ 
0.509%
\end{array}%
$ & $%
\begin{array}{c}
0.733 \\ 
0.770 \\ 
0.748%
\end{array}%
$ & $%
\begin{array}{c}
0.977 \\ 
0.854 \\ 
0.881%
\end{array}%
$ & $%
\begin{array}{c}
0.837 \\ 
0.925 \\ 
0.925%
\end{array}%
$ & $%
\begin{array}{c}
0.404 \\ 
0.402 \\ 
0.489%
\end{array}%
$ & $%
\begin{array}{c}
0.805 \\ 
0.876 \\ 
0.898%
\end{array}%
$ & $%
\begin{array}{c}
0.871 \\ 
0.974 \\ 
0.799%
\end{array}%
$ & $%
\begin{array}{c}
1.000 \\ 
0.929 \\ 
0.926%
\end{array}%
$ \\ 
$0.0$ & 
\begin{tabular}{l}
$0$ \\ 
$0.10$ \\ 
$0.25$%
\end{tabular}
& $%
\begin{array}{c}
0.541 \\ 
0.530 \\ 
0.503%
\end{array}%
$ & $%
\begin{array}{c}
0.907 \\ 
0.896 \\ 
0.866%
\end{array}%
$ & $%
\begin{array}{c}
0.939 \\ 
0.933 \\ 
0.910%
\end{array}%
$ & $%
\begin{array}{c}
0.995 \\ 
0.992 \\ 
0.989%
\end{array}%
$ & $%
\begin{array}{c}
0.407 \\ 
0.497 \\ 
0.519%
\end{array}%
$ & $%
\begin{array}{c}
0.902 \\ 
0.967 \\ 
0.905%
\end{array}%
$ & $%
\begin{array}{c}
1.000 \\ 
0.995 \\ 
0.894%
\end{array}%
$ & $%
\begin{array}{c}
0.930 \\ 
0.982 \\ 
0.962%
\end{array}%
$ & $%
\begin{array}{c}
0.546 \\ 
0.398 \\ 
0.668%
\end{array}%
$ & $%
\begin{array}{c}
0.874 \\ 
0.873 \\ 
0.870%
\end{array}%
$ & $%
\begin{array}{c}
0.954 \\ 
0.940 \\ 
1.000%
\end{array}%
$ & $%
\begin{array}{c}
1.000 \\ 
0.998 \\ 
0.993%
\end{array}%
$%
\end{tabular}

\bigskip
\end{center}

\textbf{Table 3. Empirical rejection frequencies for the test for }$%
H_{0}:X_{t}$\textbf{\ is nonstationary. }

\begin{center}
\newpage

\begin{tabular}{cl|llll|llll|llll}
&  & \multicolumn{4}{|c|}{$e_{t}\sim N\left( 0,1\right) $} & 
\multicolumn{4}{|c|}{$e_{t}\sim t_{2}$} & \multicolumn{4}{|c}{$e_{t}\sim
t_{1}$} \\ 
& $T$ & $250$ \  & $500$ \  & $1000$ & $2000$ & $250\ $ & $500\ $ & $1000$ & 
$2000$ & $250\ $ & $500\ $ & $1000$ & $2000$ \\ 
$\left( \varphi ,\sigma _{b}^{2}\right) $ &  &  &  &  &  &  &  &  &  &  &  & 
&  \\ 
$\left( 0.2,3.6190\right) $ &  & $0.051$ & $0.058$ & $0.053$ & $0.053$ & $%
0.054$ & $0.053$ & $0.059$ & $0.053$ & $0.054$ & $0.054$ & $0.058$ & $0.060$
\\ 
$\left( 0.3,3.5556\right) $ &  & $0.053$ & $0.058$ & $0.053$ & $0.053$ & $%
0.056$ & $0.060$ & $0.057$ & $0.053$ & $0.054$ & $0.056$ & $0.057$ & $0.053$
\\ 
$\left( 0.4,3.4460\right) $ &  & $0.054$ & $0.058$ & $0.056$ & $0.055$ & $%
0.057$ & $0.053$ & $0.060$ & $0.060$ & $0.057$ & $0.051$ & $0.051$ & $0.055$
\\ 
$\left( 0.5,3.3390\right) $ &  & $0.053$ & $0.062$ & $0.054$ & $0.055$ & $%
0.055$ & $0.056$ & $0.054$ & $0.054$ & $0.056$ & $0.057$ & $0.051$ & $0.055$
\\ 
$\left( 0.6,3.2245\right) $ &  & $0.052$ & $0.062$ & $0.052$ & $0.057$ & $%
0.055$ & $0.060$ & $0.058$ & $0.053$ & $0.054$ & $0.060$ & $0.054$ & $0.057$
\\ 
$\left( 0.7,3.1310\right) $ &  & $0.052$ & $0.059$ & $0.053$ & $0.054$ & $%
0.055$ & $0.059$ & $0.058$ & $0.054$ & $0.053$ & $0.054$ & $0.059$ & $0.054$
\\ 
$\left( 0.8,2.8650\right) $ &  & $0.055$ & $0.065$ & $0.056$ & $0.057$ & $%
0.058$ & $0.060$ & $0.056$ & $0.057$ & $0.052$ & $0.057$ & $0.052$ & $0.059$
\\ 
$\left( 0.9,2.6815\right) $ &  & $0.053$ & $0.062$ & $0.054$ & $0.058$ & $%
0.057$ & $0.060$ & $0.053$ & $0.052$ & $0.054$ & $0.057$ & $0.060$ & $0.060$
\\ 
$\left( 1.0,2.4440\right) $ &  & $0.054$ & $0.059$ & $0.056$ & $0.059$ & $%
0.055$ & $0.058$ & $0.058$ & $0.053$ & $0.057$ & $0.058$ & $0.060$ & $0.053$%
\end{tabular}

\bigskip
\end{center}

\textbf{Table 4. Empirical rejection frequencies for the test for }$%
H_{0}:X_{t}$\textbf{\ is nonstationary. The cases considered in the table
are boundary cases in which }$E\ln \left\vert \varphi +b_{0}\right\vert =0$%
\textbf{\ - thus, }$X_{t}$\textbf{\ is nonstationary, and the figures in the
table should be interpreted as the power of the test. }\newpage

\begin{center}
\begin{tabular}{lllllc}
\textbf{Series} & \textbf{Unit} & \textbf{Frequency} & \textbf{SA} & \textbf{%
Sample} & $T$ \\ 
&  &  &  &  &  \\ 
GDP & Billions of dollars & Quarterly & YES & Q1 1947-Q3 2017 & $283$ \\ 
CPI & Percentage & Monthly & NO & Feb 1913-Nov 2017 & $1259$ \\ 
Ind. Prod. & Index (2012=100) & Monthly & YES & Jan 1919-Nov 2017 & $1187$
\\ 
M2 & Billions of dollars & Weekly & YES & 03/11/80-11/12/17 & $1937$ \\ 
Unemployment & Percentage & Monthly & YES & Jan 1948-Nov 2017 & $839$ \\ 
T-Bill & Percentage & Daily & NO & 21/12/12-20/12/17 & $1304$%
\end{tabular}

\bigskip
\end{center}

\textbf{Table 5. Data description of the series employed in the empirical
exercise; the column headed SA refers to whether data are seasonally
adjusted or not. }

\begin{center}
\newpage

\begin{tabular}{lllllllll}
\textbf{Series} & $\widehat{\varphi }$ & $\widehat{\sigma }_{b}^{2}$ & $ERS$
& $KPSS$ & \multicolumn{2}{c}{$%
\begin{array}{c}
Q\left( 0.05\right)  \\ 
H_{0}:X_{t}\text{ stationary}%
\end{array}%
$} & \multicolumn{2}{c}{$%
\begin{array}{c}
Q\left( 0.05\right)  \\ 
H_{0}:X_{t}\text{ nonstationary}%
\end{array}%
$} \\ 
&  &  &  &  & (level) & (first-diff) & (level) & (first-diff) \\ 
GDP & $0.9988$ & $-1.38\times 10^{-4}$ & $\ 96.10$ $\ \left( ^{\ast }\right) 
$ & $0.293\left( ^{\ast }\right) $ & $0.0000\left( ^{\ast }\right) $ & $%
0.9514$ & $0.9526$ & $0.0000\left( ^{\ast }\right) $ \\ 
CPI & $1.0002$ & $8.19\times 10^{-3}$ & $\ 45.66$ $\ \left( ^{\ast }\right) $
& $0.709\left( ^{\ast }\right) $ & $0.0000\left( ^{\ast }\right) $ & $0.9460$
& $0.9528$ & $0.0000\left( ^{\ast }\right) $ \\ 
Ind. Prod. & $0.9995$ & $2.32\times 10^{-4}\left( ^{\ast }\right) $ & $N/A$
& $N/A$ & $0.0000\left( ^{\ast }\right) $ & $0.9448$ & $0.9468$ & $%
0.0000\left( ^{\ast }\right) $ \\ 
M2 & $0.9999$ & $-6.64\times 10^{-7}$ & $133.82\left( ^{\ast }\right) $ & $%
0.356\left( ^{\ast }\right) $ & $0.0000\left( ^{\ast }\right) $ & $0.9486$ & 
$0.9482$ & $0.0000\left( ^{\ast }\right) $ \\ 
Unemployment & $0.9884$ & $2.39\times 10^{-3}\left( ^{\ast }\right) $ & $N/A$
& $N/A$ & $0.0000\left( ^{\ast }\right) $ & $0.9472$ & $0.9444$ & $%
0.0000\left( ^{\ast }\right) $ \\ 
T-Bill & $1.0021$ & $7.18\times 10^{-5}$ & $\ 66.17\left( ^{\ast }\right) $
& $0.861\left( ^{\ast }\right) $ & $0.0000\left( ^{\ast }\right) $ & $0.9500$
& $0.9500$ & $0.0000\left( ^{\ast }\right) $%
\end{tabular}

\bigskip
\end{center}

\textbf{Table 6. Outcomes of estimation and testing for the null of strict
stationarity. We have also reported the WLS\ estimators of }$\varphi $%
\textbf{\ and of }$\sigma _{b}^{2}$\textbf{\ (}$\widehat{\varphi }$\textbf{\
and }$\widehat{\sigma }_{b}^{2}$\textbf{), as studied in \citet{HT16}; the
symbol \textquotedblleft }$\left( ^{\ast }\right) $\textbf{%
\textquotedblright\ next to the values of }$\widehat{\sigma }_{b}^{2}$%
\textbf{\ denotes rejection of the null of no coefficient randomness
expressed as }$H_{0}:\sigma _{b}^{2}=0$\textbf{\ (we refer to \citealp{ht16}
for the theory of estimation and the test). As mentioned in the paper, we
have computed two popular unit root tests: the Elliot-Rothenberg-Stock test
(reported in the column ERS) and the KPSS test. The former has been carried
out under the hypothesis of a constant and a trend, using the Bartlett
kernel to estimate the spectral density and choosing the related bandwidth
via the criteria discussed in \citet{andrews91}. The same specifications
were used for the KPSS\ test. In the last four columns, we have reported the
values of }$Q\left( \alpha \right) $\textbf{\ (for }$\alpha =0.05$\textbf{)
for the data in levels and first differences, considering both testing for
the null of stationarity and the null of nonstationarity. In all tests
considered, the symbol \textquotedblleft }$\left( ^{\ast }\right) $\textbf{%
\textquotedblright\ indicates rejection of the null hypothesis. }

\end{landscape}

\end{document}